\documentclass[12pt]{amsart}
\usepackage[all,cmtip]{xy}
\usepackage{times, pdfpages, graphicx}
\usepackage{comment}
\usepackage{exscale}
\usepackage{epsfig}

\usepackage{amsthm}
\usepackage{amsfonts}

\usepackage{graphics,psfrag,graphicx,subfigure,epsfig, xypic}
\usepackage{amsmath,amsfonts,amssymb,amstext,amsthm,amscd,mathrsfs}
\usepackage[english]{babel}
\usepackage[all]{xy}
\usepackage{multicol}
\usepackage{rotating}

\usepackage{amsmath, amsfonts, amssymb, amscd, enumerate}
\usepackage{color}
\usepackage{tikz-cd}
\usepackage[all]{xy}

\theoremstyle{plain}
 \newtheorem{theo}{Theorem}[section]

\theoremstyle{plain}
\newtheorem{corollary}[theo]{Corollary}
\newtheorem{proposition}[theo]{Proposition}
\newtheorem{definition}[theo]{Definition}
\newtheorem{remark}[theo]{Remark}
\newtheorem{example}[theo]{Example}

\newcommand{\beq}{\begin{equation}}
\newcommand{\eeq}{\end{equation}}

\renewcommand{\i}{\iota}

\newcommand{\C}{\mathbb{C}}

\newcommand{\R}{\mathbb{R}}
\renewcommand{\H}{\mathbb{H}}
\newcommand{\Z}{\mathbb{Z}}

\newcommand{\bN}{\mathbb{N}}

\newcommand{\bS}{\mathbb{S}}

\newcommand\U{\mathrm{U}}

\newcommand{\cI}{\mathcal{I}}

\newcommand{\cS}{\mathcal{S}}

\newcommand{\ra}{\rightarrow}

\renewcommand{\square}{\kern1pt\vbox
{\hrule height 0.6pt\hbox{\vrule width 0.6pt\hskip 3pt
\vbox{\vskip 6pt}\hskip 3pt\vrule width 0.6pt}\hrule height0.6pt}\kern1pt}

\DeclareMathOperator\Id{Id}

\renewcommand\Re{\operatorname{Re}}
\renewcommand\Im{\operatorname{Im}}

\newcommand{\res}{{\rm res}\,}
\newcommand{\id}{{\rm id}\,}

\renewcommand{\Re}{{\rm Re}}
\renewcommand{\Im}{{\rm Im}}

\newcommand{\ol}{\overline}

\renewcommand{\mod}{\operatorname{mod}}
\newcommand{\be}{\begin{equation}}
\newcommand{\ee}{\end{equation}}

\def\<#1,#2>{\langle\,#1,\,#2\,\rangle}
\newcommand{\arr}{\begin{array}{rlll}}
\newcommand{\ea}{\end{array}}
\newcommand{\bea}{\begin{eqnarray}}
\newcommand{\eea}{\end{eqnarray}}
\newcommand{\bean}{\begin{eqnarray*}}
\newcommand{\eean}{\end{eqnarray*}}

\def\sideremark#1{\ifvmode\leavevmode\fi\vadjust{
\vbox to0pt{\hbox to 0pt{\hskip\hsize\hskip1em
\vbox{\hsize3cm\tiny\raggedright\pretolerance10000
\noindent #1\hfill}\hss}\vbox to8pt{\vfil}\vss}}}

\newcounter{ssig}
\newcounter{ttig}
\begin{document}

\title[$*$-logarithm]{On a definition of logarithm of quaternionic functions }
\author{Graziano Gentili}
\address{DiMaI,     Universit\`a di Firenze, Viale Morgagni 67/A,\ Firenze, Italy}
\email { graziano.gentili@unifi.it }
\author{Jasna Prezelj}
\address{Fakulteta za matematiko in fiziko Jadranska 21 1000
  Ljubljana, Slovenija, UP FAMNIT, Glagolja\v ska 8, Koper Slovenija}
\email { jasna.prezelj@fmf.uni-lj.si}
\author {Fabio  Vlacci}\address{DiSPeS Universit\`a di Trieste Piazzale Europa 1,\ Trieste,
  Italy} \email{ fvlacci@units.it} \thanks{}
\begin{abstract}
For a slice--regular quaternionic function $f,$ the classical exponential function $\exp f$ is not slice--regular in general.
An alternative definition of exponential function, the $*$-exponential $\exp_*$, was given: if $f$ is a slice--regular function, then $\exp_*(f)$ is a slice--regular function as well. The study of a $*$-logarithm $\log_*(f)$ of a slice--regular function $f$ becomes of great interest for basic reasons, and
is performed in this paper. The main result shows that the existence of such a $\log_*(f)$ depends only on the structure of the zero set of the vectorial part $f_v$ of the slice--regular function $f=f_0+f_v$, besides the topology of its domain of definition.
We also show that, locally, every slice--regular nonvanishing
function has a $*$-logarithm and, at the end, we present an example of a
nonvanishing slice--regular function on a ball which does not admit a
$*$-logarithm on that ball.
\end{abstract}


\makeatletter
\@namedef{subjclassname@2020}{%
  \textup{2020} Mathematics Subject Classification}
\makeatother


\maketitle

\section{Introduction}
Let $\H$ be the skew field of quaternions and let us denote the $2$-sphere of imaginary units of $\H$ by $\bS=\{q\in \H : q^2=-1\}$  . Consider the natural exponential function $\exp \colon \H \to \H\setminus\{0\}$ defined by the classical power series:
\begin{equation}\label{classicexp}
\exp(q)=\sum_{n=1}^{+\infty}\frac{q^n}{n!}
\end{equation}
In the case of quaternions, a satisfactory definition of a (necessarily local) inverse of this exponential function - the logarithm and its different branches - is not a simple  task, together with the question of the continuation of the logarithm along curves lying in $\H\setminus\{0\}$ (see \cite{GA, paper-2, paper-3} and references therein).

Let $\Omega\subseteq \H$ be an axially symmetric domain (see Definition \ref{ASD}), and consider the class $\mathcal{SR}(\Omega)$ of all $\H$-valued slice--regular functions defined in $\Omega$ (see, e.g.,\cite{GS}). These functions have proven to be naturally suitable to play the role of holomorphic functions in the quaternionic setting, and have originated a theory that is by now quite rich and well developed (see, e.g., \cite{GSS} and references therein). Slice regular functions present several peculiarities, mainly due to the noncommutative setting of quaternions; among these peculiarities, the facts that pointwise product and composition of slice--regular functions do not produce slice--regular functions in general. The definition of the $*$-product typical of the algebra of polynomials with coefficients in a non commutative field can be extended to the class of slice--regular functions on  an axially symmetric domain $\Omega\subseteq \H$, which naturally becomes an algebra. As for composition, if $f\colon\H \to \H$ is a slice--regular function, even
\[
\exp(f(q)) = \sum_{n=1}^{+\infty}\frac{f(q)^n}{n!}
\]
turns out not to be slice--regular in general. The $*$-product helps in this situation to find an exponential function which maintains slice--regularity, defined (with obvious notations) as
\begin{eqnarray}\label{*exp}
\exp_*(f(q)) = \sum_{n=1}^{+\infty}\frac{f(q)^{*n}}{n!}
\end{eqnarray}
This $*$-exponential, defined and studied in~\cite{CSS} and further investigated in~\cite{AdF}, has many interesting properties typical of an exponential-type function.

In this paper we investigate the existence of a slice--regular  logarithm $\log_*(f)$ for
a slice--regular function $f$. This activity finds a deep motivation in the study of quaternionic Cousin problems, that the authors are performing and that will be the object of a forthcoming paper.

We will now briefly outline the path that this paper follows for the tuning of a slice--regular logarithm. Recall that any slice--regular
function $f$ defined on an axially symmetric domain $\Omega$ can be uniquely  written as
\[f=  f_0+ if_1+jf_2+kf_3=f_0 + f_v \]
where $\{1,i,j,k\}$ denotes the standard basis of $\H$, where $$f_0(q) = \frac{f(q)+ \overline{f(\bar{q})}}2$$ is the \emph{scalar} part of $f$ and  $$f_v:= f - f_0$$ its \emph{vectorial part}. The vectorial part $f_v$ of $f$ plays a fundamental role in the definition of $\log_*$. Indeed, with the adopted notations we have 
\begin{eqnarray}\label{espression for exp}
  \exp_*(f)&=& \exp_*(f_0+f_v)=\exp(f_0)\exp_*(f_v) \notag \\
  &=&\exp(f_0)\left(
  \sum_{m\in\mathbb{N}} \dfrac{(-1)^m(f_v^s)^m}{(2m)!}+
\sum_{m\in\mathbb{N}} \dfrac{(-1)^m(f_v^s)^m}{(2m+1)!}f_v
 \right)\\ \notag
 &=& \exp(f_0)\left(\cos(\sqrt{f_v^s})+\sin(\sqrt{f_v^s})\dfrac{f_v}{\sqrt{f_v^s}}\right)
\end{eqnarray}
 when  the symmetrization $f_v^s:=
f_1^2+ f_2^2+ f_3^2$ of $f_v$ does not vanish, and where the definitions of $\cos,$  $\sin$ and $\sqrt{f_v^s}$ are the natural ones. A less algebraic, but maybe more enlightening, point of view is the following. To better understand the computation of $\exp_*(f_v)$ let us notice that, since it holds, 
$$f_v*f_v=-f^s_v=-f_v*f^c_v$$ then, outside the zero set of $f^s_v$, we have
\[
\dfrac{f_v}{\sqrt{f_v^s}}*\dfrac{f_v}{\sqrt{f_v^s}}= \dfrac{-f^s_v}{f_v^s}=-1
\]
identically. Therefore the vectorial function $$\dfrac{f_v}{\sqrt{f_v^s}}$$ can be given the role of an
 imaginary unit, and therefore
$$
\exp_*(f_v)=\exp_*(\dfrac{f_v}{\sqrt{f_v^s}}{\sqrt{f_v^s}}) = \cos(\sqrt{f_v^s})+\sin(\sqrt{f_v^s})\dfrac{f_v}{\sqrt{f_v^s}}
$$
All this said, we begin by focusing our study of the solutions $f$ of the equation $$\exp_*f = g$$ to the case of  $\exp_*f =1 $
on an axially symmetric domain $\Omega$ whose intersection $\Omega_I$ with $\R+I\R\cong \C_I$ is ``small" for any $I\in \bS$. We  then proceed to the
definition of a local  $*$-logarithm for any slice--regular function on
such a domain.
As one may expect, once the function $\log_*g$ is defined, we can also define the real powers of $g$, like for example
    \begin{equation}\label{radix}
      \sqrt[s]{g} := \exp_*{\left(\frac{1}{s}\log_{*}g\right)},
 \end{equation}
    for  all $s\in\mathbb{R}, s > 0$.

It turns out that the structure of the zeroes of the
vectorial part $g_v$ of the slice--regular function $g\colon\Omega \to \H$ in question plays a key role. Roughly speaking, the set $Z(g_v)$ of nonreal and nonspherical zeroes of the vectorial part $g_v$ of $g$ (shared with the entire \emph{vectorial (equivalence) class $[g_v]$} and for this reason denoted $Z([g_v])$, see Definition \ref{vec_class})  determines the right conditions for the existence of the
$*$-logarithm of $g$ in such a domain $\Omega$.
In the chosen setting, a slice--regular function $g\colon \Omega \to \H$ belongs to the vectorial class $[0]$ if and only if its vectorial part $g_v$ is equivalent to the null function in $\Omega$, that is, if and only if $g$ belongs to the same vectorial class of its scalar part $g_0$. 
This situation is particularly fortunate for our study, as explicitly suggested by Formula~\eqref{espression for exp}.


The set of all slice--regular functions $g\in \mathcal{SR}(\Omega)$ which are in the vectorial class $[0]$ is  denoted by  $\mathcal{SR}_{[0]}(\Omega) =\mathcal{SR}_{\R}(\Omega).$ In general,  
$\mathcal{SR}_{\omega}(\Omega)$ will denote the set of slice--regular functions $g\in \mathcal{SR}(\Omega)$
whose vectorial parts  $g_v$ are in the  class $\omega$ (see Section \ref{vectclass}).

For the existence of a $*$-logarithm of a function $g\in \mathcal{SR}(\Omega)$, a sort of slicewise simple--connectedness of the axially symmetric domain $\Omega$ is required (but is not in general a sufficient condition): indeed we will require that each of the, at most two, connected components of $\Omega_I=\Omega \cap \C_I$ is simply connected for one (and hence for all) $I\in \bS$. Such a domain $\Omega$ will be called \emph{a basic domain}.  If $W\subseteq \H$ is any subset, then we will set the notation $\bS W:=\{sw : s\in \bS, w\in W\}$ and use it henceforth.

The main theorem of this paper, stated below and proved in Section~\ref{6} (Subsection~\ref{6.1}) together with some of its consequences, identifies sufficient conditions for the existence of a $*$-logarithm of a function $g\in \mathcal{SR}(\Omega)$ with respect to the different structures of the vectorial class $[g_v]$ and of
its zero set $Z([g_v])$.
 \begin{theo}\label{general} Let $\Omega\subseteq \H$ be a basic domain and let $g \in \cS\mathcal{R}_{\omega}({\Omega})$ be a nonvanishing function.  Then it holds:
 \begin{enumerate}
  \item[$(a)$] if $\omega = [0]$,   a necessary and sufficient condition
    for the  existence of a $*$-logarithm of $g$ on $\Omega$, $\log_*g \in \mathcal{SR}_{[0]}(\Omega),$
    is $$g(\Omega \cap \R) \subset (0,+\infty);$$
   \item[$(b)$] if $\omega \ne [0]$, then  if  $Z(\omega) = \varnothing$ or if \ $\bS Z(\omega) = \Omega$   there are no conditions, and a $*$-logarithm of $g$ on $\Omega$, $\log_*g \in \mathcal{SR}_{\omega}(\Omega)$, always exists;
    \item[$(c)$] if $\omega \ne [0]$ and $Z(\omega)$ is  discrete,  a sufficient  condition for the  existence of
      a $*$-logarithm of $g$ on $\Omega$, $\log_*g \in \mathcal{SR}_{\omega}(\Omega)$, is the validity of both inclusions
      \begin{equation}\label{realimage}
        \sqrt{g^s}(\Omega \cap \R) \subset (0,+\infty)
      \end{equation}
      and
    \begin{equation}\label{counterex}
      \frac{g_0}{\sqrt{g^s}}(\Omega) \subset \H \setminus (-\infty,-1].
    \end{equation}
 \end{enumerate}
where $g^s=g_0^2+g_v^s$ denotes the symmetrization of $g$.
\end{theo}

Now, if the functions $\mu, \nu \in \mathcal{SR}_{\R}(\H)$ are defined by the identities
\[\mu(z^2) = \cos z \qquad \mathrm{and}\qquad   \nu(z^2) = \frac{\sin z}{z}, \] then the last formula in (\ref{espression for exp})
can be rewritten as \[\exp_*f = \exp(f_0)  (\mu(f_v^s) + \nu(f_v^s)f_v).\]
Moreover, for any $I \in \bS$ the mapping
\[\mu_I : \C_I \setminus \{k^2 \pi^2\ : \ k \in \bN \cup \{0\} \} \ra \C_I \setminus\{1,-1\}\] turns out to be  a covering map (see subsection \ref{properties_mu}). In this setting, we can obtain the  second main result of this paper which appears in Section~\ref{6} (Subsection~\ref{6.2}): Theorem~\ref{isolated}. It produces a formula for the $*$-logarithms of a non vanishing slice--regular function $g$, defined on a basic domain with no real points and whose vectorial class $[g_v]$ has only one (non real) zero.

In  the last section, we also show that for the following function
$$
  g(z) = -1 + z^2i + \sqrt{2} z j + k,
$$
which is nonvanishing on the ball $B^4(0,1.1)$, there is no slice--regular logarithm  globally defined in the entire $B^4(0,1.1).$
Indeed, this function $g$  meets the hypotheses of Theorem \ref{general} $(c),$ but does not fulfil the stated sufficient conditions (\ref{realimage}) and (\ref{counterex}).

While  preparing  the  final  draft  of  this  paper, we became aware that results  similar to ours, but   suggested  by  different motivations  and  involving  different  techniques,  were  obtained  by  Altavilla and de Fabritiis and are now posted on {\tt arXiv} (\cite{AdF1}).\\

\section{Preliminary results}
 Given any quaternion $z \not\in
 \mathbb{R},$ there exist (and are uniquely determined) an imaginary
 unit $I\in \bS$, and two real numbers, $x,y$  $y>0$, such that
 $z=x+Iy$. With this notation, the conjugate of $z$ will be $\bar z :=
 x-Iy$ and $|z|^2=z\bar z=\bar z z=x^2+y^2$.
 Each $I\in \bS$ generates (as a real algebra) a copy
 of the complex plane denoted by $\mathbb{C}_I = \R + I\R $. We call such a complex
 plane a {\em slice}.  The upper half-plane in $\mathbb{C}_I$, namely the set $\C_I^+:=\{x+yI\in \mathbb{C}_I\ :\ y>0\}$ will be called a {\em leaf}.

\begin{definition}\label{ASD} A domain $\Omega$ of $\mathbb H$ will be called \emph{axially symmetric}\footnote{Some authors use the term ``circular.''} if
\[
\Omega=\bigcup_{x+Iy\in \Omega} x+\mathbb Sy
\]
i.e., if for all $x,y \in \R$ and all $I\in \mathbb S$, we have that $x+Iy\in \Omega$ implies that the entire $2$-sphere $x+\mathbb Sy$ is contained in $\Omega$.
\end{definition}

The proof of the following facts is straightforward:

\begin{proposition}
Let $\Omega\subseteq \H$ be an axially symmetric domain. For all $I\in \mathbb S$, we have that \[
\Omega=\bigcup_{x+Iy\in \Omega_I} x+\mathbb Sy
\]
Moreover, for all $I\in \mathbb S$, the set $\Omega_I\subseteq R+I\R$ is invariant under conjugation, i.e., $\Omega_I=\overline{\Omega_I}$.
\end{proposition}

A class of natural domains of definition for slice--regular functions is the following one.

\begin{definition}\label{slice_domain} A domain $\Omega$ of $\mathbb H$ is called a \emph{slice domain} if, for all $I\in \mathbb S$, the subset
$\Omega_I$ is a domain in $\R+I\R$ and if $\Omega\cap\R\neq \varnothing$.  If, moreover, $\Omega$ is axially symmetric, then it is called a {\em symmetric} slice domain.
\end{definition}

On the other hand, slice functions (see [GP]) are naturally defined  on axially symmetric domains which are not necessarily slice domains.

\begin{definition} An axially symmetric domain $\Omega$ of $\mathbb H \setminus \R$ is called a \emph{product domain}.
\end{definition}

If $\Omega\subseteq \H$ is an axially symmetric domain, then for (one and hence for) all $I\in \mathbb S$, the set $\Omega_I$ is an open subset of $\C_I$ such that: either it is a connected set that intersects $\R$, or  it has two symmetric connected components separated by the real axis, swapped by the
conjugation. In the former case, $\Omega$ is an axially symmetric slice domain; in the latter case $\Omega$ is a product domain.

\begin{proposition}
Let $\Omega\subseteq \H$ be an axially symmetric domain. Then $\Omega$ is either a symmetric slice domain or it is a product domain.
\end{proposition}

The following class of domains will play a key role in this paper.

\begin{definition}\label{gjf} A domain $\Omega$ of $\mathbb H$ is called a {\em basic} {\em  domain} if it is axially symmetric and if, for (one and hence for) all $I\in \mathbb S$, the single connected component or both the connected components of  $\Omega_I$ are simply connected.
A basic domain is also a \emph{basic neighborhood} of any of its points.
\end{definition}

The following examples show that being a 
  simply connected domain and being a basic domain are distinct notions
  in general.

  \begin{example}{\em
    For any given pair of positive real numbers $0<r<R$, the axially symmetric domain $A_{r,R}=\{q\in\mathbb{H}\ :\ r<|q|<R\}$ is simply connected but the domain of the slice $\C_I$ obtained as
    $A_{r,R}\cap \mathbb{C}_I$ is not simply connected for any $I\in\mathbb{S}$. Hence$A_{r,R}$ is not a basic domain.}
\end{example}
  \begin{example}{\em
    The axially symmetric domain $\mathbb{H}\setminus\mathbb{R}$
    is not simply connected, but the intersection of $\mathbb{H}\setminus\mathbb{R}$ with any slice $\mathbb{C}_I$ has two connected components, and each one
    is simply connected. Hence $\mathbb{H}\setminus\mathbb{R}$ is a basic domain}.
    \end{example}

We will now recall a unified definition of the class of slice regular functions on axially symmetric domains, valid both for slice domains and for product domains (see, e.g., \cite{GP1}). If $\i^2 = -1,$  consider the complexification $\H_\C =\H+\i \H,$ of the skew field $\H$ and set $x+\i y \mapsto x-\i y$ to be the natural involution of $\H_\C$. For any $J\in \bS$, let the map

$$\phi_J\colon \H_\C \to \H$$ be defined by $$\phi_J(x+\i y)=x+Jy$$
Notice that the map $\phi_J$, when restricted to $\R_\C\cong \R+\i \R\cong \C$, is an isomorphism between $\C$ and $\R+J\R=\C_J$.

If $\Omega \subseteq \H$ is an axially symmetric domain, and if $i$ denotes the imaginary unit of $\C$,  then the intersection $\Omega_{i}=\Omega\cap (\R+i\R)=\Omega\cap \C$ defines a domain of the complex plane that is invariant under complex conjugation, i.e., is such that $\Omega_{i}=\overline{\Omega_{i}}$. With respect to the established notations, the subset $\Omega_{\i}=\{x+\i y \in \H+\i \H : x+iy \in \Omega_i\}$ is called the \emph{image of $\Omega_i$ in $\H_\C$}, and is invariant under involution, i.e., $\Omega_{\i}=\overline{\Omega_{\i}}$.
We are now in a position to recall the following definitions.

\begin{definition}\label{stem} Let $\Omega \subseteq \H$ be an axially symmetric open set, let $\Omega_{i}=\Omega\cap (\R+i\R)$ and let $\Omega_{\i}$ be the image of $\Omega_i$ in $\H_\C$.

A function $F \colon \Omega_{\i} \to\H_\C$ is called a \emph{stem function} if $F(\bar z) = \overline{F(z)}$ for all $z \in \Omega_{\i}$. For each stem function $F \colon \Omega_{\i} \to \H_\C$, there exists a unique $f \colon \Omega \to \H$ such that the diagram
$$\begin{tikzcd}
\Omega_{\i}\arrow{d}{\phi_J}\arrow{r}{F}&\H_\C\arrow{d}{\phi_J}\\
\Omega\arrow{r}{f}&\H
\end{tikzcd}$$
commutes for all $J \in \bS$. The function $f$ is called the \emph{slice function} induced by $F$ and denoted by $\mathcal{J}(F)$.


Let $f=\mathcal J(F ), g=\mathcal J (G)$ be the slice functions induced by the stem functions $F, G$ respectively. The \emph{$*$-product} of $f$ and $g$ is defined as the slice function $f*g:=\mathcal J(FG)$.

\end{definition}
We will use a definition of slice regularity (and $*$-product) that involve stem functions, and that is valid for any axially symmetric domain of $\H$. When restricted to symmetric slice domains, it coincides with the Definition 1.2 of slice regularity initially presented in \cite{GS}.

\begin{definition}\label{regularitygp} Let $\Omega \subseteq \H$ be an axially symmetric open set.

A slice function $f \colon \Omega \to \H$, induced by a stem function $F \colon \Omega_{\i} \to \H_\C$, is called \emph{slice--regular} if $F$ is holomorphic.
The set of all slice--regular functions on $\Omega$ is denoted by $\mathcal{SR}(\Omega).$

A  slice
function $f\colon \Omega \to \H$ is said to be {\em slice--preserving} if and only if
$\forall I \in \bS, \forall z\in\Omega_I:=\Omega\cap \mathbb{C}_I$ we have that
$f(z)\in \C_I$. The set of all slice--regular
functions, which are slice--preserving in $\Omega,$ will be denoted as
$\cS\mathcal{R}_{\mathbb{R}}(\Omega)$.
\end{definition}

The next proposition recalls two well known technical results that will be extensively used in the sequel (see, e.g., \cite{GP1}).


\begin{proposition}\label{regularitygp} Let $\Omega \subseteq \H$ be an axially symmetric open set, and let $f, g\in \mathcal{SR}(\Omega)$ be two slice--regular functions. Then
\begin{enumerate}
\item[(a)] the $*$-product $f*g$ is a slice--regular function on $\Omega$;
\item[(b)] if $f$ is slice--preserving, then $f*g=fg=g*f$\ , i.e, the $*$-product coincides with the pointwise product.
\end{enumerate}
\end{proposition}

Let us now define the \emph{imaginary unit function} 
 \[\cI \colon \H \setminus \R \to \mathbb{S}\, \]
by setting $\cI(q) = I$ if $q \in \mathbb{C}_I.$  The function $\cI$ is slice--regular and slice--preserving, but it is not an open mapping and  it is not defined on any slice domain.

Consider now an axially symmetric open set  
$\Omega$  and $f \in \mathcal{SR}(\Omega).$ We have already defined the splitting $f = f_0 + f_v,$ where the scalar part
$f_0$ of $f$ is a slice--preserving function.
\begin{definition}
  The function $f\in \mathcal{SR}(\Omega)$ is a {\em vectorial function} if $f = f_v.$
  The set of  vectorial functions on $\Omega$ will be denoted by
  $\cS\mathcal{R}^v(\Omega)$. We have $\mathcal{SR}(\Omega) = \mathcal{SR}_{\R}(\Omega)\oplus\mathcal{SR}^v(\Omega).$
  \end{definition}
Given a standard basis of $\H,$ the vectorial part can be decomposed further
(\cite{CGS}, Proposition 3.12, compare \cite{AdF}, Proposition 2.1):
\begin{proposition}\label{frm}
  Let $\{1, i, j, k\}$ be the standard basis of $\H$ and assume $\Omega$ is an  axially symmetric domain of $\H$.
  Then the map
\[ (\cS\mathcal{R}_{\mathbb{R}} (\Omega))^4\ni  (f_0,\ f_1,\ f_2,\ f_3 ) \mapsto f_0 + f_1 i + f_2 j + f_3 k \in \cS\mathcal{R}(\Omega)\]
is bijective.
\end{proposition}

\noindent In the sequel, all bases of $\mathbb{H}\cong \mathbb{R}^4$ will be orthonormal (and positively oriented)
with respect to the standard scalar product of $\mathbb{R}^4$.
Proposition \ref{frm} implies that, given any $f, g \in \mathcal{SR}(\Omega)$, there exist and are unique
$f_0,\ f_1,\ f_2,\ f_3, g_0,\ g_1,\ g_2,\ g_3 \in \cS\mathcal{R}_{\mathbb{R}}(\Omega)$ such that
\begin{eqnarray*}
f = f_0 + f_1 i + f_2 j + f_3 k=f_0+f_v\\
g=g_0+g_1i+g_2j+g_3k=g_0+g_v
\end{eqnarray*}
With the above given notation, if we call \emph{regular conjugate} of $f$ the function
\[f^c=f_0-f_v\] then we have \[f_0= \frac{f+f^c}{2}.\]
Furthermore, using Definition~\ref{stem} and Proposition~\ref{regularitygp}, we obtain the following expression for the $*$-product of $f$ and $g$:
\begin{equation}\label{*product}
f*g := f_0g_0-f_1g_1-f_2g_2-f_3g_3+f_0g_v+g_0f_v+ \dfrac{f_v*g_v-g_v*f_v}{2}
\end{equation}
%
We now set
\[f^s:=f_0^2+f_1^2+f_2^2+f_3^2=f*f^c=f^c*f\]
and call $f^s$ the \emph{symmetrization} of $f$.\\

\section{Basic properties of the exponential}

  If $\exp(q)$ is
the (quaternionic) exponential mapping
defined in \eqref{classicexp}, then
for every $k \in \Z,$ we define its restriction to the cylinder $\{q\ :\ \Im(q) \in  \bS (k\pi,(k+1)\pi)   \}$
to be
  \[\exp_k\colon \{q\ :\ \Im(q) \in  \bS (k\pi,(k+1)\pi)  \} \ra \H \setminus \R.\]
For any $k\in\Z$ the function  $\exp_k$  is a bijective slice--regular slice--preserving function  with a slice--regular and slice--preserving inverse, namely
$$
      \log_{k}(q) = \log|q| + \cI(q)\,\arg_{\cI(q),k}(q)\,,
      $$
{where  $\arg_{\cI(q),k} \in (k\pi,(k+1)\pi)$ denotes the argument of $q$ in the complex plane $\C_{\cI(q)}.$  }
      The mapping $\log_0$ is  called  the {\em principal branch} of the logarithm and can be extended to \[ \log_0\colon \H \setminus (-\infty, 0] \ra \{q\ :\ \Im(q) \in  \bS [0,\pi)   \} \] as the inverse of the extension of
 \[\exp_0 : \{q\ :\ \Im(q) \in  \bS [0,\pi)   \} \ra  \H \setminus (-\infty, 0].\]

Let turn our attention to the problem of computing the logarithm of a function $g$, defined on a domain
$\Omega $ of  
$\mathbb{H}$. For any continuous function $g\colon\Omega \to \mathbb{H}\setminus (-\infty,0]$, one can define
\[f:=\log_0 \circ \,g,\]
so that the diagram
$$\xymatrix
   { & \mathbb{R}\times \mathbb{S}[0,\pi) \ar[dr]^{\exp} & \\
     \Omega \ar[rr]^{g} \ar[ur]^{f} & & \mathbb{H}\setminus (-\infty,0]
   }
$$
commutes.
 In these hypotheses, for any $z\in \Omega,$ we have the equality
$\exp(f(z))=g(z)$ by definition, but even if $g$ is slice--regular, no
regularity on the function $f$ can be argued.   If, in addition $\Omega,$ is axially symmetric and $g\in \mathcal{SR}_\R(\Omega)$
is a slice--regular and slice--preserving function, then
$f$ is a well-defined slice--regular and slice--preserving
function too.
Indeed, (see Proposition~\ref{regularitygp})
 the equality
$\exp_*(f)=\exp(f)=g$ holds on $\Omega$  for $f = \log_0 \circ\, g$ and we say that the function $f$ is a logarithmic function of $g$ (in $\Omega$).


We have thus shown

\begin{proposition}
Let  $\Omega\subseteq\H$ be a symmetric slice domain.  If $g\in  \mathcal{SR}_{\R}(\Omega)$ is such that
\[
g(\Omega) \subset \mathbb{H}\setminus (-\infty,0]
\]
then the function
\[
f= \log_0 \circ\, g
\]
is the (slice--regular and slice--preserving) principal logarithm of $g$.
\end{proposition}

Let us point out that if $f \in \mathcal{SR}_{\R}(\Omega)$,  with $\Omega\subseteq \H$  any symmetric slice
domain, then
 $(\exp_* f)(x_0)= (\exp f)(x_0)> 0$ for any $x_0 \in \Omega \cap \R.$
Hence the condition
\begin{equation}\label{cond1}
g(\Omega \cap \R) \subset (0,+\infty)
\end{equation}
is a necessary condition for a slice--preserving function $g\in \mathcal{SR}_{\R}(\Omega)$ to have a slice--preserving logarithm (see also \cite{AdF}).

\section{$*$-Exponential of a quaternionic function} \label{rem1}

In this section we shortly recall some results from \cite{AdF}, which are necessary to explain our definition of $*$-logarithm.

 The $*$-exponential map of a slice--regular
function $f\in\cS\mathcal{R}(\Omega)$, with $\Omega$ axially symmetric domain, is defined for any $z\in \Omega$ as in \eqref{*exp} by
\[\exp_*(f(z))=\sum_{k\geq 0} \dfrac{f(z)^{*k}}{k!}\]
in such a way that $\exp_*(f)\in\cS\mathcal{R}(\Omega).$
The equality $\exp_*(f+g) = \exp_*f * \exp_* g$ does not hold in general as stated in Theorem~\ref{*commute} (see also Theorem 4.14 in \cite{AdF}), which we premise a crucial definition to.
\begin{definition}\label{lin_dep}
Let $f_v \in \cS\mathcal{R}^v(V)$ and $g_v \in \cS\mathcal{R}^v(V),$ where $V\subset \H$
is an axially symmetric domain in $\mathbb{H}$. 
We say that $f_v$ and $ g_v$
are linearly dependent over $\cS\mathcal{R}_{\mathbb{R}}(V)$ 
if and only
if there exist $a,b\in\cS\mathcal{R}_{\mathbb{R}}(V)$, with $a$ or $b$ not identically zero in $V$,
such that $af_v+bg_v=0$ in $V$.
If  $V\subset \H$ is an axially symmetric open set in $\mathbb{H}$, then 
 $f_v$ and $ g_v$
are linearly dependent over $\cS\mathcal{R}_{\mathbb{R}}(V)$ 
if and only if they are linearly dependent over $\cS\mathcal{R}_{\mathbb{R}}(V_{\lambda} )$ for each connected component $V_{\lambda}$ of $V$.
\end{definition}

\begin{remark}\label{onlyrealzeros}
  Real isolated zeroes and isolated spherical zeroes can be factored out of a slice regular function (see, e.g., \cite{GSS, GP1}).
  As a consequence for any vectorial function
$f_v\in \cS\mathcal{R}^v(\Omega)$ on an axially symmetric open set $\Omega$ and for every  axially symmetric open set $V \Subset 
\Omega$, there exists a non identically zero, slice--regular and slice--preserving function $a\in \cS\mathcal{R}_{\mathbb{R}}(\Omega)$ such that
$$
f_v=a\widetilde f_v
$$
with $\widetilde f_v\in \cS\mathcal{R}^v(\Omega)$ having neither real nor spherical zeroes on $V$. Of course $f_v$ and $\widetilde f_v$ are linearly dependent over $\cS\mathcal{R}_{\mathbb{R}}(\Omega)$.
\end{remark}

\begin{theo}\label{*commute}
  Assume that the
{axially symmetric}
  domain $\Omega$  intersects the real axis (i.e., it is a symmetric slice domain).  Take $f, g \in  \mathcal{SR}(\Omega)$. If
\begin{equation}\label{formula1}
\exp_*(f+g)=\exp_*(f)*\exp_*(g)
\end{equation}
then either
\begin{itemize}
\item[(i)] $f_v$ and $g_v$ are linearly dependent over $\mathcal{SR}_{\R}(\Omega)$ or
\item[(ii)] there exist $n, m, p \in  \mathbb{Z}\setminus \{0\}$ such that $f^s=n^2\pi^2, $ $g^s=m^2\pi^2,$

$2
(f_1g_1+f_2g_2+f_3g_3)=(p^2 -n^2-m^2)\pi$
and $n + m \cong$ p mod 2.
\end{itemize}
Vice versa, if either (i) or (ii) are satisfied, then (\ref{formula1}) holds.
\end{theo}

Hence equality \eqref{formula1} holds if 
there exist
$a ,b \in\cS\mathcal{R}_{\mathbb{R}}(\Omega)$ such that $af_v+bg_v=0$  with $a\not \equiv 0$ or $b\not\equiv 0$.
In particular this implies
\[\exp_*(0)=\exp_* (f ) * \exp_* (-f ) \equiv 1\]
so that, for every $f\in\cS\mathcal{R}(\Omega)$, the slice--regular function $\exp_* (f )$ is a nonvanishing function in $\Omega$.
If $f=f_0+f_v$,  then
\begin{equation}
\label{expexp}\exp_*(f)=\exp(f_0)\exp_*(f_v).\end{equation}
Moreover
 \[
 \exp_*(f^c)=(\exp_*(f))^c
 \]
whence
\begin{eqnarray*}
(\exp_*(f))^s &=& \exp_*(f)*(\exp_*(f))^c\\&=&\exp_*(f)*\exp_*(f^c)=\exp_*(f+f^c)\\ &=&\exp(2 f_0)
\end{eqnarray*}
\noindent and, from \eqref{expexp},
\begin{equation} \label{structure}
  \exp_*(f)=\exp(f_0)\left(\sum_{m\in\mathbb{N}} \dfrac{(-1)^m(f_v^s)^m}{(2m)!}+\sum_{m\in\mathbb{N}} \dfrac{(-1)^m(f_v^s)^m}{(2m+1)!}f_v
  \right).
\end{equation}
Following \cite[Remark 4.8]{AdF} we will use the notations
\begin{eqnarray}\label{munu}
  \mu(z):= \sum_{m\in\mathbb{N}} \dfrac{(-1)^m z^m}{(2m)!}, \quad \nu(z):=\sum_{m\in\mathbb{N}} \dfrac{(-1)^mz^m}{(2m+1)!}.
\end{eqnarray}
 Both functions $\mu$ and $\nu$ are entire  slice--regular and slice--preserving functions in $\H$, in symbols $\mu,\nu\in\cS\mathcal{R}_{\mathbb{R}}(\H)$.
Furthermore,
\begin{equation}\label{munu2}
\mu(z^2) = \cos z \qquad \mathrm{and}\qquad   \nu(z^2) = \frac{\sin z}{z}
\end{equation}
where, in general,
\begin{equation}\label{cos}
  \cos_* (f)=\sum_{m\in\mathbb{N}} \dfrac{(-1)^m(f)^{*(2m)}}{(2m)!}\quad\mathrm{and}\quad \sin_* (f)=\sum_{m\in\mathbb{N}} \dfrac{(-1)^mf^{*(2m+1)}}{(2m+1)!}
\end{equation}
for $f\in \cS\mathcal{R}(\Omega)$.
Notice that also $\cos_*$ and $\sin_*$ are entire slice--regular and slice--preserving functions in $\H.$
More in detail (see again \cite[Corollary 4.7]{AdF}),
given a basic domain $\Omega$ and a slice--regular function $f\colon\Omega\to \H$, such that
 $f_v^s$ is not identically zero
 and $f_v$ has only real or spherical zeroes, 
then, in $\Omega$,
\begin{equation}\label{polar}
  \exp_*(f)=\exp(f_0)\left(\cos(\sqrt{f_v^s})+\sin(\sqrt{f_v^s})\dfrac{f_v}{\sqrt{f_v^s}}\right),
\end{equation}
where $\sqrt{f_v^s}$ is defined in the obvious way, being $f_v^s$ a slice--preserving function. 
 Indeed, we will refer to \eqref{polar}
as the  \emph{polar representation} for $\exp_*(f)$. The reader can find more details about the definition of square roots in
Proposition 3.1 and Corollary 3.2 in \cite{AdF} (see also \cite{paper-2}).

\subsection{Properties of the function $\mu$ }\label{properties_mu}
Let us first list some properties of the function $\mu$, defined by \eqref{munu}, which are
essential to define the logarithm of a slice--regular function. Since we have the identity
$\mu(q^2)= \cos(q),$ for any $q\in\mathbb{H}$, we first define the branches $\mu_k$ of $\mu$ using  the branches  of the inverse of the function $\cos,$ i.e. the inverses of
\[\cos_k : \{q\ :\ \Re(q) \in (k\pi,(k+1)\pi)\} \ra \H ((-\infty,-1] \cup [1,+\infty)),\]
denoted by $\arccos_k.$
To this end consider first the domains
\begin{itemize}
 \item $D_0:= \{q\ :\ \Re(q) \in [0,\pi) \},$ $D_{-1}:= \{q\ :\ \Re(q) \in (-\pi,0]\},$
 \item $D_{k}:=\{q\ :\ \Re(q) \in (k\pi,(k+1)\pi)\}$ for $k \in \Z \setminus\{0,-1\}.$
\end{itemize}
Notice that domains $D_k, k \ne 0,-1,$ lie entirely either in the right  half-space $\{q\ :\ \Re(q) > 0\}$ or in the left half-space
$\{q\ :\ \Re(q) <0\},$ so the squaring map  $p_2,$  $p_2(q) = q^2,$ is injective on each $D_k$ and hence bijective onto $p_2(D_k)$ with an inverse $\sqrt{\phantom{a}}.$

For all $k \in \Z$ define the domains $M_k,\tilde{M}_k$ to be
$$
  M_{k}:= p_2(D_k), \, \tilde{M}_k:= \mu(M_k) = \cos (D_k),
$$
and observe that
\begin{itemize}
\item $0 \in M_0= M_{-1},$
\item $\tilde{M}_0 = \tilde{M}_{-1} = \H \setminus (-\infty,-1],$
\item $\tilde{M}_k = \H \setminus ((-\infty,-1] \cup [1,+\infty)), k \ne 0,-1.$
\end{itemize}
By definition, for each $k \in \Z,$ the diagram
$$\xymatrix
   { & \ar[dl]_{p_2} D_k \ar[dr]^{\cos} & \\
     M_k  \ar[rr]^{\mu_k} & & \tilde{M}_k
   }
$$
commutes. The choice of domains $D_k,\, k \ne 0,-1,$ is such that both $\cos$ and $p_2$ are bijective, hence so is $\mu_k.$
To see that also $\mu_0$ and $\mu_{-1}$ are bijective it remains to show that they are  bijective when restricted to the imaginary axis. In this case, since both $\cos$ and $p_2$ are even, we have for $q \in  \Im(\H)$
\[\cos(q) = \cos(-q) \mbox{ and }  p_2(q) = p_2(-q).\]
%
Moreover, for each $I \in \bS,$ $k = 0,-1,$ the restrictions
\[\cos: I[0,+\infty)   \ra [1,+\infty) \] and
\[p_2: I[0,+\infty)  \ra (-\infty,0] \]
are injective, which implies that the induced maps $\mu_0, \mu_{-1}$ are bijective.

The points $k \pi, k \in \Z$, are branching points for the complex cosine, which implies that the points $k^2\pi^2$ are  branching points for $\mu,$ except the point $0,$ which is contained in $M_0$ and where $\mu_0'(0) = -1/2 \ne 0.$\\

We can summarize these considerations in the following
\begin{proposition} For each $k$ the function $\mu_k:=\mu|_{M_k}: M_k \ra \tilde{M}_k$ is bijective with  the inverse $\mu^{-1}_k.$
  In particular, $\mu_0(0)=1$ and
  the function $\mu_0$ maps a neighbourhood of $0$ bijectively to a neighbourhood of $1.$
The mapping
$$
  \mu \colon \H \setminus \{k^2\pi^2, k \in \bN \cup\{0\}\} \ra \H \setminus \{\pm 1\}
$$
is a {\em slice--covering map}, i.e.
$$
  \mu_I \colon\C_I \setminus \{k^2\pi^2, k \in \bN \cup \{0\}\} \ra \C_I \setminus \{\pm 1\}
$$
is a covering map for every $I \in \bS.$  Furthermore, any map $\mu_I$ extends to a local diffeomorphism across the point $0.$
\end{proposition}

It turns out that for $k \ne 0, -1$ we have
\[
   \arccos_{k}:= \sqrt{\phantom{a}} \circ \mu_k^{-1}: \tilde{M}_k \ra D_k,
\]
and, for $k = 0,-1$ we have
\[
   \arccos_{k}:= \sqrt{\phantom{a}} \circ \mu_k^{-1}\colon \tilde{M}_k \setminus [1,+\infty) \ra D_k \setminus \{q\ :\ \Re(q) = 0\}.
\]



\section{Globally defined vectorial class}\label{vectclass}

Formula (\ref{formula1})  shows how crucial it is
for two slice--regular functions to have linearly dependent vectorial parts.
This motivates the following

\begin{definition}\label{vec_class}
  Let $f_v \in \cS\mathcal{R}^v(U)$ and $g_v \in \cS\mathcal{R}^v(U'),$ where $U,U'\subset \H$
are axially symmetric domains  in $\mathbb{H}$ such that $U\cap U'\neq\varnothing$.
Take $p\in U\cap U'$; we say that $f_v$ and $g_v$ are \emph{equivalent at $p$}, in symbols
 $f_v \sim_p g_v$, if there exist an  axially symmetric neighborhood of $p$, $V_p \subset U\cap U'$, such that  $f_v$ and $g_v$
are linearly dependent over $\cS\mathcal{R}_{\mathbb{R}}(V_p)$ in $V_p$. We will denote by $[f_v]_p$ the $\sim_p$ equivalence class whose representative is $f_v$.

\end{definition}

It is easy to verify that the relation $\sim_p $ is an equivalence relation at each point $p$; The definition above immediately implies that if $f_v \sim_p g_v$ then $f_v \sim_q g_v$ for every $q \in \bS p=:\bS_p.$ Moreover:

\begin{remark}\label{minimalrep}
For each equivalence class $[f_v]_p$\,  we can choose a local representative $\widetilde f_v$ having neither real nor spherical zeroes (see Remark~\ref{onlyrealzeros}).
\end{remark}

\begin{definition}
By $\mathcal{V}_p$ we denote the set of all  $\sim_p $ equivalence classes of vectorial functions at $p,$ namely
\[ \mathcal{V}_p:= \{[f_v]_p: f_v \in \mathcal{SR}^v(U), \mbox{$U$ axially symmetric neighborhood of $p$} \}.\]
\end{definition}

\begin{definition} Let $U $ be an axially symmetric open set and
\[
  \mathcal{V}_U:=\{\mathcal{V}_p, p \in U \}
\]
be  the set of all  equivalence classes of vectorial functions with respect to equivalence relations $\sim_p,$ with $p \in U.$  A {\em vectorial class} $\omega_U$ on  $U$ is defined to be any function $$\omega_U\colon U \ra \mathcal{V}_U$$  such that:
\begin{itemize}
\item for all $p\in U$, it holds $\omega_U(p) \in \mathcal{V}_p$;
\item if $p,q \in U,$  if $\omega_U(p) = [f_v]_p$ with $f_v \in \cS\mathcal{R}^v(V_p)$ for an axially symmetric domain $V_p\subset U$ containing $p$, if $\omega_U(q) = [g_v]_q$ with
$g_v \in \cS\mathcal{R}^v(V_q)$ for an axially symmetric domain $V_q\subset U$ containing $q$, then $[f_v]_{\tilde{p}} = [g_v]_{\tilde{p}}$ for all $\tilde{p} \in V_p\cap V_q.$
\end{itemize}
We denote by ${\mathcal V}(U)$ the set of all vectorial classes  over $U.$

If  $f_v \in \mathcal{SR}^v(U)$ then it obviously defines the vectorial class on $U$
 $$p \mapsto [f_v]_p, p \in U$$  which we denote
by $[f_v]_U$ and call  {\em principal  vectorial class (associated to $f_v$)} on $U.$

\end{definition}

Notice that ${\mathcal V}(U)$ is not a ring over $\mathcal{SR}_{\R}(U)$ (it is not possible
to define the sum of two classes); furthermore, if $[f_v]_U =[g_v]_U$  and $\tilde{f}_v$ and $\tilde{g}_v$ are representatives on $U$ without real or spherical zeroes, then $\tilde{f}_v^s$ identically zero in $U$ implies $\tilde{g}_v^s$ identically zero in $U$.

\begin{definition}
Let $U, U'\subset \H$ be two axially symmetric open sets such that $U\cap U'\neq
\varnothing$. If $f_v \in \cS\mathcal{R}^v(U)$ and $g_v \in
\cS\mathcal{R}^v(U')$ 
are linearly dependent
over $\cS\mathcal{R}_{\mathbb{R}}(U\cap U')$ in $U \cap U'$,
then they define a vectorial class $[f_v \vee g_v]_{U\cup U'} \in \mathcal{V}(U\cup U')$
by

\[\left\{\begin{array}{ccc}
\lbrack f_v \rbrack   &\mathrm{in}& U\setminus U'\\
\lbrack f_v\rbrack  = \lbrack g_v\rbrack  &\mathrm{in}& U\cap U'\\
\lbrack g_v \rbrack &\mathrm{in}& U'\setminus U\\
\end{array}\right.\]
\end{definition}

\begin{definition}
Let  $V \subset U \subset \mathbb{H}$ be  axially symmetric open sets and  let $\omega_U\in\mathcal{V}(U)$ be a vectorial class on $U$. 
The restriction  morphism
\[\res_{V,U} \colon \mathcal{V}(U) \ra \mathcal{V}(V)\]
is defined by
\[ \res_{V,U}(\omega_U):=\omega_U|_{U\cap V}=:\omega_V.\]
\end{definition}

\begin{proposition}\label{PreSH} The collection $\{U,{\mathcal V}(U)\}$  of vectorial classes over all axially symmetric domains  $U \subset \H$
together with the families of restriction  morphisms
$\res_{V,U} \colon \mathcal{V}(U) \ra \mathcal{V}(V), V\subset U,$
 is a presheaf.
\end{proposition}
\begin{proof}
It is immediate that  $\res_{U,U} = \id_{\mathcal{V}(U)}$.  It is also immediate that
  $\res_{W,V} \circ \res_{V,U} = \res_{W,U}$ holds for
axially symmetric domains $W \subset V \subset U,$  since vectorial classes are functions.
\end{proof}

\begin{proposition}\label{SH}  The presheaf $\{U,\mathcal{V}(U)\}$ is a sheaf and will be denoted by ${\mathcal V}.$
\end{proposition}
\begin{proof}
  Let $U$ be an axially symmetric domain and $\varpi_U, \omega_U \in \mathcal{V}_U.$ Let $\{U_{\alpha}\}_{\alpha \in \Lambda}$ be an open covering of $U$ with axially symmetric open sets.
  \begin{itemize}
  \item[(i)] Locality. If we have $\varpi_{U_{\alpha}} = \omega_{U_{\alpha}}$ for all $\alpha \in \Lambda$,
    then by definition $\varpi_U = \omega_U.$ \\
  \item[(ii)]  Gluing. Let the vectorial classes $\omega_{\alpha, U_{\alpha}},$ $ \alpha \in \Lambda$ be such that
     \[\omega_{\alpha,U_{\alpha}}|_{U_{\alpha} \cap U_{\beta}} = \omega_{\beta,U_{\beta} }|_{U_{\beta} \cap U_{\alpha}}, \alpha, \beta \in \Lambda.\]  The function defined by
   \[ (\omega_U)|_{U_{\alpha}} :=  \omega_{\alpha,U_{\alpha}}, \, \alpha \in \Lambda, \]
 is a vectorial class  on $U.$
  \end{itemize}
\end{proof}

\begin{remark}
Vectorial classes  ${\mathcal V}(U)$ are sections of the sheaf $\mathcal{V}_U.$
\end{remark}

Let $\Omega$ be an axially symmetric domain and $f_v$
a vectorial function on $\Omega.$ Then, being slice--regular, its
symmetrization $f_v^s$ is either identically\footnote{This doesn't imply that $f_v$ is identically zero. Consider for example
  $f_v \in\cS\mathcal{R}(\H \setminus \R)$  defined as $f_v(x + Iy) = Ii + j;$
  then $f_{v}^s=I^2 + 1 \equiv 0,$  $f_v \not \equiv 0$ (and $f_v$ has a zero on every sphere).}
$0$  or has isolated real or spherical zeroes.

\begin{proposition} Let $\Omega$ be an axially symmetric domain, $f_v = f_1 i + f_2 j + f_3 k$
a vectorial function on $\Omega$ and assume that $f_v^s$ is not identically zero on $\Omega.$
Let $z_0$ be a real zero of $f_v^s.$ Then it is a real zero of $f_v$ and  there exists $k > 0$ such that
$$
   (q - z_0)^{-k} f_v  =: g_v,
$$
$g_v \in \cS\mathcal{R}_{[f_v]}(\Omega)$ and $g_v(z_0) \ne 0.$ Similarly, if  $f_v$ has a spherical zero  $\bS_{z_0}=\{a + Ib\ :\ I\in\mathbb{S}\}$ of multiplicity $k > 0,$ then
$$
  (q^2 - 2q\Re(z_0) + |z_0|^2)^{-k} f_v   =: g_v,
$$
$g_v \in \cS\mathcal{R}_{[f_v]}(\Omega)$ and $g_v(q) \ne 0$ for all $q \in \bS_{z_0},$ except maybe at one point.
\end{proposition}

\begin{proof}
First notice  that $z_0$ is a real zero of $f_v \not\equiv 0$  if and only if it is a common zero of $f_l, l = 1,2,3.$
If $z_0$ is a real zero of $f_v^s \ne 0$
then $f_1^2(z_0)+ f_2^2(z_0) + f_3^2(z_0) = 0$ which implies that $z_0$ is a common
zero of all the components of $f_v$ of multiplicity $k$ for some $k \in \bN,$ since $f_l(z_0) \in \R, l = 1,2,3.$  Therefore we may factor out
a slice--preserving factor $(q - z_0)^{k}$
from $f_1,f_2,f_3$ and hence the function $(q - z_0)^{-k}f_v $ is nonvanishing on a neighbourhood of $z_0.$
In other words, one can locally write $f_v = \lambda w,$ where $w$ does not
have real zeroes and $\lambda \not \equiv 0$ is a slice--preserving function.
If $f_v$ has a spherical zero  $\bS_{z_0}=\{a + Ib\ :\ I\in\mathbb{S}\}$,
then $f_l(a +Ib) =a_l + I b_l,\, l = 1,2,3$ for any $I\in\bS$.  For $i =I$ we have
that $f_2(z_0)j + f_3(z_0)k = x_1 j + x_2 k$ and $f_1(z_0)i=a_1i-b_1,$ hence the
condition $f_1(z_0)i + f_2(z_0)j + f_3(z_0)k = 0$ implies $a_1 = b_1 = 0$ and, analogously, $a_l = b_l = 0$ for $l = 2,3,$ hence
 $f_1,f_2,f_3$ all have $\bS_{z_0}$ as a spherical zero.
If the spherical zero is of multiplicity $k$,
then we can factor out a term $ (q^2 - 2q\Re(z_0) + |z_0|^2)^{k}$ from $f_l$, with  $l=1,2,3$.
\end{proof}

\begin{definition} Let $\omega$ be a  vectorial class on an axially symmetric domain $\Omega.$
  Define $${\cS\mathcal{R}}_{\omega}(\Omega)=\{g \in {\cS\mathcal{R}}(\Omega)  : [g_v] _p\in \omega(p), \, \forall p\in \Omega\}. $$ For the case $\omega=[0]$, notice that by definition ${\cS\mathcal{R}}_{[0]}(\Omega)={\cS\mathcal{R}}_{\R}(\Omega)$.
\end{definition}

If $f,g \in \mathcal{SR}_{\omega}(\Omega)$ then also $f*g = g*f \in \mathcal{SR}_{\omega}(\Omega),$ because the last term in the Formula (\ref{*product}) vanishes. In particular, since $\mathcal{SR}_{\R}(\Omega)\subseteq \mathcal{SR}_{\omega}(\Omega)$ for any $\omega,$  if $f \in \mathcal{SR}_{\omega}(\Omega)$ and
$g \in \mathcal{SR}_{\mathbb{R}}(\Omega)$ then $f*g \in \mathcal{SR}_{\omega}(\Omega)$.

Remark~\ref{minimalrep} suggests now the following

\begin{definition} Let $\Omega$ be an axially symmetric domain and let
$\omega \in {\mathcal V}(\Omega).$ Let $U \subset \Omega$ be an axially symmetric open set and let $w \in \cS\mathcal{R}_{\omega}(\Omega)$
be the vectorial part of a slice--regular function. Then
 $w$ is called {\em minimal on  $U$} if it has neither real nor spherical zeroes on $U.$
\end{definition}

We have shown that in the case $f^s_v \not\equiv 0,$ spherical and real zeroes of the vectorial part are precisely the
common zeroes of the components of $f_v.$
The vectorial function $w(z) =z^2 i + \sqrt{2}z j + k$ is an example of a
minimal representative; it has an isolated zero on the unitary sphere $\bS$, namely $z_0=\frac{k-i}{\sqrt 2} j,$  and  its
symmetrization $w^s(z) = (z^2 + 1)^2$
vanishes on $\mathbb{S}$. Notice furthermore,  that $z^2 + 1$ is not a
common factor of the components of $w$.

For all $f_v\in\mathcal{SR}^v(\Omega)$, the factorization $f_v = \lambda w$ with $w\in \mathcal{SR}_{[f_v]}(\Omega)$ minimal  and $\lambda \in \mathcal{SR}_{\R}(\Omega)$  is unique up to a
multiplication by a slice--preserving nonvanishing function. If $w_{\alpha},w_{\beta}$ are two minimal representatives of the same vectorial class on an axially symmetric subset $U \subset \Omega,$  then
$a w_{\alpha} = b w_{\beta}$ for some $a,b \in \mathcal{SR}_{\R}(U)$ and by minimality both $a$ and $b$ are nonvanishing on $U;$ moreover the zero sets of $w_{\alpha}$ and $w_{\beta}$ coincide.
Therefore, given a vectorial class $\omega$ on an axially symmetric domain $\Omega$, we can define the zero set $Z(\omega)$ of  $\omega.$
\begin{definition}\label{zerosets2}
Let $\Omega$ be an axially symmetric domain and let
$\omega \in {\mathcal V}(\Omega)$ be a vectorial class.

If $\omega \neq 0$, let $w$ be a minimal representative of $\omega$ on an axially symmetric open set $U \subset \Omega.$ Define
$Z(\omega) \cap \U = w^{-1}(0).$ Then the zero set $Z(\omega)$ of  $\omega$ is defined to be the union of all zeroes $w^{-1}(0)$ where $w$ runs over  minimal representatives of $\omega$ on open axially symmetric subsets $U$ of $\Omega.$

If $\omega = [0]$, then we define $Z([0]) = \varnothing.$
\end{definition}

\begin{proposition}\label{zerosets1}
 Let $\Omega$ be an axially symmetric domain and let
$\omega \in {\mathcal V}(\Omega)$ be a vectorial class. 
 If
  $w$ is a local minimal representative of $\omega$ on an axially symmetric domain $U \subset \Omega$, then
  \begin{itemize}
  \item [(i)]  if $w^s \not\equiv 0,$ then $Z(\omega)\subset \Omega$ is a discrete set of non real quaternions;
  \item[(ii)]  if $w^s \equiv 0$ but $w \not \equiv 0,$  we have $\bS Z(\omega) = \Omega,$ there is precisely one zero of $Z(\omega)$ on each sphere and, moreover, $\Omega \subset \H \setminus \R.$
  \end{itemize}
  \end{proposition}

\begin{proof}
Let $w$ be a local minimal representative of $\omega \ne [0]$ on a basic domain $U \subset \Omega.$ Then $w^s$ is slice--preserving and hence it is either identically equal to $0$ or has isolated real or spherical zeroes (or no zeroes).
If $w^s$ is not identically equal to $0$, the same holds for any other minimal representative by the identity principle and then obviously the set $Z(\omega)$ is either discrete or empty.

Assume that $w^s \equiv 0$ but $w\not\equiv 0.$ Recall that, for any other representative $\tilde{w}$ we have $\tilde{w}^s \equiv 0$, by the identity principle.
The identity principle implies that $\Omega\subset \H \setminus \R$  is a product domain. Indeed, if $\Omega$ is a slice domain, then on the real axis the symmetrization $w^s$ is a sum of
squares of real numbers and hence, if it is identically $0$, then by the
identity principle also $w \equiv 0$ in an axially symmetric domain containing $\Omega\cap \R$, and hence in the entire slice domain $\Omega$; contradiction.

 Now, $w^s \equiv 0$ on the product domain $\Omega$ implies that $w$ has a zero on each sphere, and can have neither a sphere of zeroes nor a real zero, since it is a minimal representative of $\omega$.

 Since $w^s = w * w^c$ and $w^c = -w,$ the equation $w^s(z_0) = 0$ implies that either $w(z_0) = 0$ or if $w(z_0) \ne 0,$ $w^c(z) = -w(z)= 0$ for $z =w(z_0)^{-1} z_0 w(z_0) \in \bS_{z_0}.$ If there were two distinct zeroes on $\bS_{z_0}$ then extension formula would imply that $w(\bS_{z_0}) = 0,$ which contradicts the assumption that $w$ is minimal.

\end{proof}

If $f_v = \lambda w$ with $w$ minimal and $\lambda$ a slice--preserving function is the (local) decomposition of $f_v$, then
$f_v^s = \lambda^2 w^s.$ If $w^s$ is nonvanishing on
a basic domain $U \subset \Omega,$ then
one can define square roots  of $f^s_v$ and $w^s$ (denoted as $\sqrt{f^s_v}$ and $\sqrt{w^s}$) (see \cite{AdF}, Proposition 3.1. and next sections) and find that $\sqrt{f_v^s}=\pm \lambda \sqrt{w^s}$.  Therefore we can state that:

\begin{proposition} \label{2normalized}

Let $\Omega$ be a basic domain, let $\omega\neq[0]$ be a vectorial class on $\Omega$ with $Z(\omega)=\varnothing$ and let $f_v \in \mathcal{SR}_\omega (\Omega)$. If $w$ is a minimal representative of $\omega$ in $\Omega$, then the \emph{normalized} vectorial function
$$f_v/\sqrt{f_v^s} \in \{ \pm w /\sqrt{w^s}\}$$ is minimal and such that
$$(f_v/\sqrt{f_v^s})^s = 1$$
in $\Omega$.
\end{proposition}
\begin{proof}
After the premises to this statement, the proof is straightforward.
\end{proof}


%
%
\section{Local definition of $\log_*$.}\label{section_cases}

We now reach the heart of the problem: if $\Omega$ is an axially symmetric domain of $\H$, given $g\in\cS\mathcal{R}(\Omega)$ not vanishing in $\Omega$ and $z \in \Omega$  an arbitrary point, find an open axially symmetric neighbourhood  $U$ of $z$ and a function $f \in \mathcal{SR}(U)$ such that
\[\exp_*f=g\, \mbox{ \ \ \ \ on }\, U.\]

The assumption that $g\in\cS\mathcal{R}(\Omega)$  is a nonvanishing  function in $\Omega$
is intrinsic with the problem, since, where defined, the function $\exp_*f$ is nonvanishing.
We will find necessary and sufficient conditions on $g$ to define a local logarithmic function of $g$.

Let us assume henceforth that $\Omega$ is a basic domain  in $\mathbb{H}.$
After writing $g=g_0+g_v$ and $f=f_0+f_v$ as in the previous section, we'll proceed by steps.

\subsection{Case 0: $g\in \cS\mathcal{R}(\Omega)$ is a constant function }\label{case0}
  To avoid confusion, the constant function $q_0$ will be denoted by $C_{q_0}$.

Consider first the case $q_0= 1.$ Then the principal branch of the logarithm can be defined, because the function $\exp_0$ is a bijection between $\{q\ :\ \Im(q) \in \bS [0,\pi)\}$ and $ \H \setminus (-\infty,0]$ and so we can define
$$
  \log_{*,0,0}:= \log_0(C_{1}) = 0,
$$
in the whole $\H.$ Choose a point $z_0 \in \H$ and let $\omega$ be any vectorial class with $z_0 \not \in \bS Z(\omega).$ Let $w$ be one of the two normalized minimal nonzero representatives of $\omega$ (see Proposition~\ref{2normalized})  defined on a basic neighbourhood $U_{z_0}$ of $z_0.$ Then,  for all $n\in \Z$, the function
\begin{equation}\label{cost1}
  \log_{*,0,2n w}(C_1):= 2 \pi n w
\end{equation}
also satisfies $\exp_*(\log_{*,0,2n w }C_1) = 1$ (see Formula \eqref{polar}).
If, moreover, $U_{z_0} \subset \H \setminus \R$  is a product domain then the imaginary unit function ${\mathcal I}$ is a well-defined
slice--preserving function
and hence  we have the possibilities
\begin{equation} \label{f2}
\begin{array}{l}
  \log_{*,m,nw}(C_1):=m \pi {\mathcal I} + n \pi  w,
\end{array}
\end{equation}
on $U_{z_0},$ where  $m,n \in \Z$ are such that $m+n\equiv 0\ (\mod 2)$. Notice that if $U_{z_0}$ is a basic slice domain, then  the only possibilities are those appearing in Formula \eqref{cost1}.

For any constant function $C_{q_0}$, $q_0 \in \H \setminus(-\infty,0], $ the situation is completely analogous, and we have
$$
  \log_{*,0,2n w}(C_{q_0}):= \log_0(q_0)  + 2n \pi  w.
$$
(for any $n\in \Z$) on a basic slice neighborhood $U_{z_0}$ of $z_0$, and
\begin{equation}
\begin{array}{l}
  \log_{*,m,nw}(C_{q_0}):= \log_0(q_0) + m \pi {\mathcal I} + n \pi  w,
\end{array}
\end{equation}
on $U_{z_0},$ where  $m,n \in \Z$ are such that $m+n\equiv 0 \ (\mod 2)$.

Consider now the constant function $C_{-1}.$ Define, for $n\in \Z$,
$$
 \log_{*,0,2n w }(C_{-1}):=(2n+1)\pi w.
$$
This function satisfies
$\exp_*(\log_{*,0,2n w}(C_{-1})) = -1$ and  on a basic product neighborhood $U_{z_0}$ of a point $z_0 \in \H \setminus \R$ we also have \begin{eqnarray*}
 \log_{*,m,n w}(C_{-1})&:=& \pi w + (m \pi {\mathcal I} + n \pi  w) = \pi w + \log_{*,m, nw} C_1,
%
\end{eqnarray*} for $n+m \equiv 0 \ (\mod 2)$.
With the notation of the previous section, for any constant function $C_{q_0}$, $q_0 \in \H \setminus\{0\}, $ we have
$$
 \log_{*,m, n w}C_{q_0} \in \mathcal{SR}_{\omega}(U).
$$

{
\begin{remark}\label{rem51}{\em
Once a slice--regular logarithm of two slice--regular functions  $g, h\in \cS\mathcal{R}(U)$ is defined in a basic domain $U$,
one can always add to each logarithm a vectorial function $2n \pi w$  (with $w$  any normalized minimal representative of a vectorial class $\omega$  in the basic domain $U$
with $Z(\omega) \cap U = \varnothing$),  but for the price of losing the property $\exp_*(\log_*(g) + \log_*(h)) = g*h .$
Indeed, notice that, for example, the equality
$$
  \exp_*(\log_{*,2m_1,2n_1w_1}(C_{1}) + \log_{*,2m_2,2n_2 w_2}(C_{1})) = 1
$$
is not necessarily valid, if $w_2 \not\in [w_1]$ (compare $(ii)$ in Theorem 4.14, \cite{AdF}).
The property $\exp_*(\log_*(g) + \log_*(h)) = g*h $ is still valid for functions  $g,h\in \exp_*^{-1}(\mathcal{SR}_{\R}(U)) \cap  (\mathcal{SR}_{\omega}(U)).$}\\
\end{remark}
}

Remark \ref{rem51} suggests to restrict our considerations to the sets
$\exp_*^{-1}(\mathcal{SR}_{\omega}(U)) \cap  (\mathcal{SR}_{\omega}(U)).$
{ According to Proposition \ref{zerosets1} and Definition \ref{zerosets2} we have the following four different possibilities with respect to the vectorial classes and the structure of their zero sets.}

\subsection{Case 1: $g\in \cS\mathcal{R}_\R(\Omega)$ is slice--preserving, i.e. $g_v \equiv 0$}\label{C1}
Let's now consider the general case of a nonvanishing slice--regular and  slice--preserving function $g = g_0.$
In this case the involved regular functions behave like holomorphic functions on each slice, but at the same time topological obstructions near the real axis complicate the problem  of finding a logarithmic function.

We assume that the necessary condition expressed by Formula ~\eqref{cond1}, i.e., $g_0(\Omega \cap \R) \subset (0,+\infty)$, holds. Then,
 since $g\neq 0,$ one can locally define a
logarithmic function of $g$ in the following way. Consider a point $z_0 \in \Omega.$

If $z_0\not \in \R$, then we have the following possibilities:
\begin{itemize}
\item $g(z_0) \in \H \setminus (-\infty,0]$, then a logarithmic function  of $g$ can be
    defined in a neighbourhood of $z_0$ since the function $\exp_0$ is a bijection between $\R
    \times \bS [0,\pi)$ and $ \H \setminus (-\infty,0]$; indeed, locally, for all $m\in \Z$, we can define
    $$
   \log_{*,2m,0}(g):=\exp_0^{-1}\circ g +2m\pi\mathcal I .
$$
\item $g(z_0) \in (-\infty,0)$. In this case a logarithmic function can be locally defined for $-g$ as in the previous point. And then we can exploit the equality:
$$
\log_{*,2m,0}(g)=\log_{*,2m,0}(-g)+\pi \mathcal I= \exp_0^{-1}\circ (-g) +(2m+1)\pi\mathcal I
$$
\end{itemize}
If $z_0\in \R$, then by hypothesis $ g(z_0) > 0$ and we have the only possibility:
$$
   \log_{*,0,0}(g):= \exp_0^{-1}\circ g
$$
since the function $\mathcal I$ cannot be defined on the real axis.
%

\begin{remark}  {\em Condition $(\ref{cond1})$ is necessary if we want the logarithm of a slice--preserving function to be slice--preserving. If not, then this condition is no longer needed. Indeed, consider any normalized minimal representative $w$ of any vectorial class defined on an axially symmetric neighbourhood $U$ of $z_0$ which is nonvanishing on $U$ and assume that $g(z_0) < 0$ for some $z_0 \in \R.$ Then
$\log_{*,0,0}(-g)$ is defined and
$$
  f = \log_{*,0,(2n+1)w}(-g) = \log_{*,0,0}(-g) +  (2n+1) \pi {w}
$$
satisfies
\begin{eqnarray*}
  \exp_*{f} &=&  - g(-1)^0 (\mu((2n +1 )^2\pi^2) + \nu(((2n+1)^2\pi^2) (2 n +1) \pi w )= \\
  &=&- g(-1)^{2n+1} = g.
\end{eqnarray*}
}
\end{remark}

\begin{remark}{\em The above considerations imply that given a nonvanishing slice--regular and slice--preserving function $g\in \cS\mathcal{R}_\R(\Omega)$
          (not necessarily satisfying condition (\ref{cond1}), that $g(\Omega \cap \R) \subset (0,+\infty)$), one can always locally define a slice--preserving logarithmic function of at least
          one of the two functions  $g, -g$ or both, depending on the domain of definition.
   }
\end{remark}



%

\subsection{Case 2: $g\in \cS\mathcal{R}(\Omega)$ with $g_v \not\equiv 0, \, g_v^s \equiv 0$}\label{C2}
Consider now $g=g_0+g_v$ such that $g$ is nonvanishing  and $g_v$ is  not
identically $0$ but $g_v^s$ is (which implies $g_0$ is nonvanishing).  Then  $\Omega$ is a product domain since otherwise  $g_v$ would be identically $0$ because of the
identity principle (Proposition \ref{zerosets1}). Therefore $\log_{*,2m,0}g_0$  can be locally defined on a basic neighbourhood $U$ of any point of $\Omega$, 
for all $m\in\mathbb{Z}$. The class $[g_v] = \omega$ does not have a normalized minimal representative, therefore in this case we use the notation
$\log_{*,m, 0 \cdot [g_v]}$ to indicate, that the resulting function is in the class $\mathcal{SR}_{\omega}(\Omega)$ but there are no periods in
any minimal representative of $[g_v].$

In general, whenever $g=g_v$ with $g_v^s\equiv0$, the equality
\[\exp_*g=\exp_* g_v= 1+g_v\]
holds, since  $g_v^{*2}=-(g_v)^s=0$ and then $g_v^{*k}$
vanishes for all $k\geq 2$. In these cases we put

\[\log_* (1 + g) = \log_{*,0,0 \cdot [g_v]}(1+g_v):=g_v=g.\]
Assume now $g=g_0+g_v$ with
$g_0$ nonvanishing and $g_v^s\equiv0$ in $\Omega$.
Then one can write $g=g_0\left(1+\dfrac{g_v}{g_0}\right)$; hence,
from $\exp_*(\log_{*,0,\cdot [g_v]}g)=g$, one concludes that

\[\log_{*,0,\cdot [g_v]}g = \log_{*,0,0} g_0 + \log_{*,0,\cdot [g_v]}(1+g_v/g_0)= \log_{*,0,\cdot [g_v]} g_0 + \dfrac{g_v}{g_0};\]
more in general,  on a product domain $U\subseteq \Omega$,
\[\log_{*,m,0}g: = \log_{*,m,0} g_0  + \dfrac{g_v}{g_0},\quad m\in\Z,\]
which completely describes all possible solutions for $\exp_*f = g$ with the given assumptions for $g$, namely
$g$ not vanishing,  $g_v\not \equiv 0$ and $g_v^s\equiv0$ in $\Omega$.

\begin{example}{\em
Consider the function
$$z=x+Iy\mapsto\Psi(x+Iy):=Ii+j;$$
clearly $\Psi=\Psi_0+\Psi_1i+\Psi_2j+\Psi_3k$ is well-defined,  slice--regular in $\Omega=\H\setminus \R$ and constant on any slice and $\Psi|_{\Omega_{-k}}\equiv 0$.
Moreover, since $\Psi_0=0$, $\Psi_1=\mathcal{I}$, $\Psi_2=1$  and $\Psi_3=0$, then $\Psi_v^s=0$. Hence
\[\exp_* \Psi= 1+\Psi.\]
Notice that $\Psi_1=\mathcal{I}\in\cS\mathcal{R}_{\mathbb{R}}(\H\setminus \R)$ and cannot be extended continuously to   $\H$.
Consider now the function $g(z)=z+\Psi(z)$; clearly $g$ is a nonvanishing
slice--regular and slice--preserving function in $\Omega=\H\setminus \R$. Furthermore, $g_0=\Id$ and $g_v=\Psi$ and so,
for any $z\in\H\setminus\R$, we have
\begin{eqnarray*}(\log_{*,k}g)(z)& =& [\log_{*,k}(\Id+\Psi)](z)=[\log_{*,k}(g_0+g_v)](z)\\
  &=& \log_{*,k} (z)  + \dfrac{\Psi(z)}{z}\\
&=&\log (|z|)+[\arg_{\cI(z)}(z)+2k\pi]\mathcal{I}(z)+\dfrac{\Psi(z)}{z},
\end{eqnarray*}
where $\log$ represents the usual real natural logarithm.}
\end{example}

\subsection{Case 3:  $g\in \cS\mathcal{R}(\Omega)$ and  $z_0\in \Omega$  such that $Z([g_v]) \cap \{\bS_{z_0}\} = \varnothing$}\label{C3}
The condition $Z([g_v]) \cap \{\bS_{z_0}\} = \varnothing$ implies the following:
either  $g_v \ne 0$ on $\bS_{z_0}$ or there is a factorization $g_v = \lambda \tilde w$
with $\tilde{w} \ne 0$ on $\bS_{z_0}.$ Hence the  function $h := \sqrt{\tilde{w}^s}$ 
is locally well-defined on a basic open neighbourhood $U$ of $z_0$ and
satisfies $h^2 =\tilde{w}^s.$ Put $\sqrt{g_v^s}:= \lambda \sqrt{\tilde{w}^s}.$
The normalized vectorial function
$$
   \frac{g_v}{\sqrt{g_v^s}}= \frac{\tilde{w}}{\sqrt{\tilde{w}^s}}=: w
$$ is thus well-defined in $U.$
   Similarly, the function $\pm \sqrt{g^s}$ is well-defined in
   $U.$ If $U$ intersects the real axis, we choose the sign
   so that $\sqrt{g^s}(U \cap \R) \subset
   (0,+\infty).$ Then $f_0:= \log_{*,0,0} \sqrt{g^s}$ is
   well-defined. If $U$ does not intersect the real axis then we
   define $f_0$ in accordance to the next formula
   \begin{equation}\label{values}
\begin{array}{ll}
  f_0 := \log_{*,2m,0} (\sqrt{g^s}), &\mbox{ if } \sqrt{g^s}(\bS_{z_0})\subset \H \setminus (-\infty,0)  \mbox{ and }
  \\
   f_0 :=\log_{*,2m+1,0} (-\sqrt{g^s}),& \mbox{ if } \sqrt{g^s}(\bS_{z_0})  \subset(-\infty,0).
   \end{array}
\end{equation}
with $m \in \mathbb{Z}.$
Notice also that the image of a sphere $\bS_{z}$ by a slice--preserving function is always a sphere centered on the real axis.

     For $f = f_0
   + f_v = \log_*g$ following Formula (\ref{polar}), we want the
   listed equalities to hold:
\begin{eqnarray*}
   \frac{f_v}{\sqrt{f_v^s}}&=& \frac{g_v}{\sqrt{g_v^s}},\\
   \cos_*{\sqrt{f_v^s}} &=& \frac{g_0}{\sqrt{g^s}},\\
   \sin_*{\sqrt{f_v^s}} &=& \frac{\sqrt{g_v^s}}{\sqrt{g^s}}.\\
\end{eqnarray*}
 For each $I \in \bS$ define the complex manifold $\Sigma_I$ to be the regular
 set $s^{-1}(1)$ for $s \colon \C_I^2 \ra \C_I,$ $s(u,v) = u^2 + v^2.$ It
 is not difficult to show that the mapping
$$
  T\colon \C_I \rightarrow \Sigma_I,\quad T(q)=(\cos q,\sin q)
$$
is a covering map  and by construction we have
$$
  G:=\left(\frac{g_0}{\sqrt{g^s}},\frac{\sqrt{g_v^s}}{\sqrt{g^s}}\right)\colon U_I \rightarrow \Sigma_I.
$$
There exist a lift $\tilde{G}$ such that the diagram
$$\xymatrix
   { & \C_I \ar[dr]^{T} & \\
     U_I \ar[rr]^{G} \ar[ur]^{\tilde{G}} & & \Sigma_I
   }
$$
commutes, i.e. $T\circ \tilde{G} = G.$

If $U_I$ is simply connected, $U \cap \R$ is an open interval and $G(U \cap \R) \subset S^1 \subset \R^2 \subset (\C_I)^2$, since all  the functions are slice--preserving. The only possibility that  both $\sin(z)$ and $\cos(z)$ are real, is, that $z$ is real. Therefore for any  lift $\tilde{G}$ the restriction  $\tilde{G}|_{U \cap \R}$ is real-valued and hence satisfies the reflection property $\tilde{G}(z) = \overline{\tilde{G}(\bar{z})}.$ If $U_I$ has connected components $U_{I,n}, n = 1,2,$ then first define the function $\tilde{G}$ on $U_{I,1}$ to be an arbitrary lift of $G|_{U_{I,1}}$ and extend the definition to $U_{I,2}$ by reflection property.
Define
$$
  \sqrt{f_v^s}:=  \tilde{G}
$$
on $U_I.$
The reflection property guarantees that $\sqrt{f_v^s}$  has a slice--preserving extension to $U.$

In the case $U \cap \R \ne \varnothing$ the final formula is
\begin{equation}\label{fff}
  \log_{*,0,2n w } g=\log_{*,0,0}(\sqrt{g^s}) + (\sqrt{f_v^s} + 2 n \pi) \frac{g_v}{\sqrt{g_v^s}}.
\end{equation}
If $U \cap \R = \varnothing$ we also have periodicity in the scalar part and the  formula is
\begin{equation}\label{fff1}
  \log_{*,m ,n w} g=\log_{*, 0, 0}(\sqrt{g^s}) + m \pi \cI + (\sqrt{f_v^s} +  n \pi) \frac{g_v}{\sqrt{g_v^s}},
\end{equation}
where $m,n\in \Z$ are such that $m+n\equiv 0 \ (\mod 2)$ and the logarithm $f_0:= \log_{*,0,0}(\sqrt{g^s})$ is  chosen in accordance with \ref{values}.

Notice that, contrary to the previous Cases 1 and 2 (and the next case, Case 4), in  Case 3 one cannot specify the ``principal branch", unless one chooses a specific point in the domain and specific normalized minimal representative.

\subsection{Case 4: $g\in \cS\mathcal{R}(\Omega)$ and  $z_0\in \Omega$  are such that $z_0 \in \bS Z([g_v])$}\label{C4}
Without loss of generality we assume that $z_0 \in Z([g_v]),$ since the logarithmic function is to be defined on a basic neighbourhood of $z_0.$
We have the following two possibilities:

$(i)$ $z_0$ is a nonreal isolated zero of $g_v,$

$(ii)$ $z_0$ is a nonreal isolated zero and  $\bS_{z_0}$ is a spherical zero of $g_v.$

Let's first consider  case $(i).$
Since $g^s(z_0) = g_0^2(z_0) \ne 0,$ we define
\begin{equation}\label{root}
  \sqrt{g^s}:=g_0 \sqrt{1 + \frac{g_v^s}{g_0^2}}
\end{equation}
with $\sqrt{\phantom{a}}$  defined using the principal branch of the logarithm (see Formula~\eqref{radix}).  The function $\sqrt{g^s}$ is a slice--preserving and slice--regular function  with  $g_0(z_0)=\sqrt{g^s}(z_0).$
This function  is well-defined in a neighbourhood of $\bS_{z_0}.$
Define
\[
  {f_v^s} :=  \mu^{-1} \left( \left(\sqrt{1 + \frac{g_v^s}{g_0^2}} \right)^{-1} \right)
\]
where $\mu^{-1}=\mu^{-1}_0$ is the inverse function of $\mu$ from a
neighbourhood of $1$ to a neighbourhood of $0,$ so that ${f_v^s}(z_0)
= 0$  (see Subsection~\ref{properties_mu}). This is equivalent to the choice of the principal branch of
$\arccos$ denoted by $\arccos_{{0}}$ indeed
$$
 f_v^s=\left(\arccos_{0} \dfrac{g_0}{\sqrt{g^s}}\right)^2.
$$ Recall that the function $\mu$
is locally invertible near $0$ because $\mu'(0)=(-1/2) \nu(0) = -1/2.$
If the function $-\sqrt{g^s}$ is chosen instead, $f_v^s$ cannot be defined
since the function $\mu$ has branching points at
\[\mu^{-1}(-1) = \{(2k+1)^2\pi^2, k \in \bN\}.\]
 \begin{remark} The
  isolated nonreal zeroes of the vectorial part force the choice of
  the function $\sqrt{g^s}$ to be such that
  $g_0(z)/\sqrt{g^s}(z) = 1$ for every zero $z$ of $g_v^s$.
\end{remark}

For the definition of $f_0$ we have to calculate a
logarithm of $\sqrt{g^s}$ depending on the two cases as in (\ref{values}):
\begin{equation*}
\begin{array}{ll}
  f_0 := \log_{*,2m,0} (\sqrt{g^s}), &\mbox{ if } \sqrt{g^s}(\bS_{z_0})\subset \H \setminus (-\infty,0)  \mbox{ and }
  \\
   f_0 :=\log_{*,2m+1,0} (-\sqrt{g^s}),& \mbox{ if } \sqrt{g^s}(\bS_{z_0})  \subset(-\infty,0).
   \end{array}
\end{equation*}
with $m \in \mathbb{Z}.$
Since  $f_v^s(z_0) = 0$  in both cases, we have $\mu(f_v^s)(z_0) = \nu(f_v^s)(z_0) = 1.$

Define the vectorial function $f_v$ to be
\[
  f_v = \exp_*(-f_0) \frac{g_v}{\nu(f_v^s)}.
\]
Then $f=f_0 + f_v$ solves $\exp_*f = g.$ The complete formula is
\begin{eqnarray}\label{case4i}
  & &\log_{*,m,0\cdot [g_v]} g\\ \nonumber &=&  \log_{*,m,0} ((-1)^m \sqrt{{g^s}}) +  \frac{g_v}{\nu\left(\mu^{-1}\left(1/ \sqrt{1 + \frac{g_v^s}{g_0^2}}\right) \right)\sqrt{g^s}},
\end{eqnarray}
where $m$  depends on the values of $\sqrt{g^s}$ as in (\ref{values})  and $\nu$ is defined in Formulas~\ref{munu} and \ref{munu2}. Notice that the period in the imaginary directions appears from the definition of the branches of logarithm for the slice--preserving part.

 If, in addition, $\bS_{z_0}$ is also a spherical zero of $g_v,$ the necessary condition, namely that $g_0(z)=\sqrt{g^s}(z)$ for every zero of $g_v^s,$ is fulfilled on the whole sphere $\bS_{z_0},$ hence the same formula applies to the case $(ii).$

\begin{remark}{\em
In the case where the zero $z_0$ has even multiplicity, the square root $\sqrt{g_v^s}$ is well-defined and  we could follow the construction for Case 3 and get  Formulae (\ref{fff}) or (\ref{fff1}); instead the vectorial part
$$
  (\sqrt{f_v^s} + 2 k \pi) \frac{g_v}{\sqrt{g_v^s}}
$$
has a pole unless we choose $k = 0.$ In addition, we must also have $\sqrt{f_v^s}(\bS_{z_0}) = 0$ and this implies that
$$
  \arccos\frac{g_0}{\sqrt{g^s}}(z_0) = 0,
$$
which at the end gives Formula (\ref{case4i}).
}
\end{remark}

\begin{remark}{\em
Let $f,g,w\in \mathcal{SR}_{\omega}(U)$ for $U$  a basic domain 
in $\H$ and let $w$ be a normalized representative of $\omega$ on $U.$
Assume that $\forall\ m,n \in \Z,$
  $\log_{*,m,nw} fg,$ $\log_{*,m,nw} f$ and $\log_{*,m,nw} g,$  all exist.  Since there is no `principal branch' in $w$, there is no reason that
  the equality
  \[ \log_{*,m,nw} fg = \log_{*,m_0,n_0w} f + \log_{*,m-m_0,(n-n_0)w} g\]
   should hold;  in general we have
   \[\log_{*,m,nw} fg = \log_{*,m_0,n_0w} f + \log_{*,m-m_0,(n-n_0)w} g + 2 k \pi  w.\]
}
\end{remark}

\section{Global definition of $\log_*$ and proof of Theorem \ref{general}}\label{6}

In this section we prove  Theorem \ref{general}, namely we consider the global problem of determining the logarithmic function of a given slice--regular function,
with the requirement that the logarithmic function defines the same vectorial class as the original function: if $\Omega$ is a basic domain of $\H$, given $g\in\cS\mathcal{R}(\Omega)$ not vanishing in $\Omega,$  find
$f\in\cS\mathcal{R}_{[g_v]}(\Omega)$ such that

\[\exp_*f=g \,\mbox{  on  }\,\Omega.\]

A classical result in complex analysis states that it is not possible to define $\log(z^2)$ on $\C \setminus \{0\}$
and hence it is also not possible to define a logarithmic function of $p_2(q)= q^2$
on $\H \setminus \{0\},$ although the function $p_2$
satisfies  the necessary condition (\ref{cond1}).

\subsection{Proof of  Theorem \ref{general}}\label{6.1}
The proof of Theorem \ref{general} is presented according to the four cases as in Section \ref{section_cases}. Here we recall the statement, before proving it.
\setcounter{section}{1}
 \begin{theo} Let $\Omega\subseteq \H$ be a basic domain and let $g \in \cS\mathcal{R}_{\omega}({\Omega})$ be a nonvanishing function.  Then it holds:
 \begin{enumerate}
  \item[$(a)$] if $\omega = [0]$,   a necessary and sufficient condition
    for the  existence of a $*$-logarithm of $g$ on $\Omega$, $\log_*g \in \mathcal{SR}_{[0]}(\Omega)= \mathcal{SR}_{\R}(\Omega),$
    is $$g(\Omega \cap \R) \subset (0,+\infty);$$
 \item[$(b)$] if $\omega \ne [0]$, then  if  $Z(\omega) = \varnothing$ or if \ $\bS Z(\omega) = \Omega$   there are no conditions, and a $*$-logarithm of $g$ on $\Omega$, $\log_*g \in \mathcal{SR}_{\omega}(\Omega)$, always exists;
    \item[$(c)$] if $\omega \ne [0]$ and $Z(\omega)$ is  discrete,  a sufficient  condition for the  existence of
      a $*$-logarithm of $g$ on $\Omega$, $\log_*g \in \mathcal{SR}_{\omega}(\Omega)$, is the validity of both inclusions
      \begin{equation}\label{realimage1}
        \sqrt{g^s}(\Omega \cap \R) \subset (0, +\infty)
      \end{equation}
      and
    \begin{equation}\label{counterex1}
      \frac{g_0}{\sqrt{g^s}}(\Omega) \subset \H \setminus (-\infty,-1].
    \end{equation}
 \end{enumerate}
 \end{theo}
\setcounter{section}{6}

\subsubsection{Proof of Theorem \ref{general} $(a)$}
The conditions in $(a)$ correspond to Case 1 presented in subsection \ref{C1}.
Assume that $\Omega$ is a basic product domain. This implies that the imaginary
unit function $\cI$ is well-defined. In each leaf $\C_I^+$ the set
$\Omega_I^+:=\C_I^+ \cap \Omega$ is simply connected. Assume that $g
\in \cS\mathcal{R}_{\R}(\Omega)$ is a nonzero function. Then
$g_I^+:=g|_{\Omega_I^+} \rightarrow \C_I$ is holomorphic and therefore it
has a holomorphic logarithm $f_I^+:=\log g_I^+.$ Because $g$ is also
slice--preserving, we can define $\log_{-I}g_{-I}^+(z)=
\ol{\log_{I}(g_{I}^+(\ol{z}))}$ and extend the logarithm to $\Omega.$
Denote this extension by $f = \log_*g.$ Similarly, the whole family of
logarithmic functions $f_{k} = \log_*((-1)^k g) + k \pi \mathcal{I}$ is also
well--defined. Notice that it is essential for this construction that
the imaginary unit function ${\mathcal I}$ exists.

Next, assume that $\Omega$ is  a basic slice
domain. Then
in each leaf $\C_I^+$ the set $\Omega_I^+:=\C_I^+ \cap
\Omega$ is simply connected and the intersection $\Omega_{\R}:=\Omega \cap \R$ is
connected. Assume that $g \in \cS\mathcal{R}_{\R}(\Omega)$ is a nonzero
function satisfying $g(\Omega \cap \R) \subset (0,\infty).$ Let
$\Omega_0$ be a connected component of $g^{-1}(g(\Omega) \cap (\H
\setminus (-\infty, 0]))$ which contains the set $\Omega_{\R}.$ Since the image $g(\Omega_0)$ does not intersect the negative real axis,
the function $f_0 = \log_{*,0,0} g$ is
well--defined on $\Omega_0$ and it is the unique logarithm as explained in Section~\ref{section_cases}.

If $\Omega = \Omega_0$ the problem is solved
so assume that $\Omega \ne \Omega_0.$
Then $\Omega_0$ is an open neighbourhood of an interval $\Omega_{\R}.$ The set $\Omega_1:= \Omega \setminus \Omega_{\R}$ is also connected and basic, but $\Omega_{1,I}:=\Omega_1 \cap \C_I$ has two connected components,  $\Omega_{1,I\pm}.$
Choose the component $\Omega_{1,I+}.$  Since it is simply connected, the function $g$ has a complex logarithm $f_+$ on $\Omega_{1,I+}.$
 On the intersection of their domains of definition (which is an open connected set), the functions $f_0$ and $f_{+}$ differ by $2 \pi k I,$ $f_0 = f_+ + 2 k \pi I.$  Redefine $f_+$ to be $f_+ + 2 k \pi I$ and define $f_{-}$ to be the Schwarz reflection of the function of $f_+.$ Since $f_0$ is slice--preserving,
$f_0(\ol{z}) = \ol{f_0(z)},$ the reflected function coincides with $f_0$ on the intersection of domains of definition and hence defines a function $f$ on $\Omega_I,$ which satisfies $f(\ol{z}) = \ol{f(z)}.$ By the extension formula, the function $f$ can be extended to a slice--preserving function on $\Omega.$

\subsubsection{Proof of Theorem \ref{general} $(b)$}

The first condition in $(b),$ $\omega \ne 0,$ $Z(\omega) =\varnothing,$ corresponds to Case 3 presented in Subsection \ref{C3}.
{
The function $g^s$ is nonvanishing,  the function $g_v^s$ has isolated real or spherical zeroes with even multiplicities and $\Omega$ is a basic domain, which are precisely the conditions of  Proposition $1.6$ in \cite{AdF}, which states, that under these conditions, the square roots $\sqrt{g^s}$ and $\sqrt{g_v^s}$ can be globally defined on $\Omega$.
}
Moreover, the normalized vectorial class
$$
  \frac{g_v}{\sqrt{g_v^s}}=:w
$$
is globally well-defined and nonzero on $\Omega.$ Therefore
Formulae (\ref{fff}) and (\ref{fff1}) are  globally valid and the logarithm exists.

The second  condition in $(b),$ $\omega \ne 0,$ $\bS Z(\omega) = \Omega$ and hence $\Omega \subset \H \setminus \R,$  corresponds to Case 2 presented in subsection \ref{C2}.
As already mentioned, when in Case 2, the basic domain $\Omega$ does not intersect the real axis and $g_0$ is not vanishing in $\Omega$.
Then, for $m \in \Z$,  one can define  $$\log_{*,2m,0 \cdot [g_v]}g:= \log_{*,2m,0} g_0  + \dfrac{g_v}{g_0}$$
since from the previous considerations $\log_{*,2m,0} g_0$
is well--defined on $\Omega.$ 

\subsubsection{Proof of Theorem \ref{general} $(c)$}
The condition in $(c),$ $\omega \not\equiv 0,$ $Z(\omega)$ is discrete, corresponds to Case 4 presented in Subsection \ref{C4}. The logarithm $\log_*\sqrt{g^s}$ exists by Case 1.
The assumptions imply that $\mu_0^{-1}(g_0/\sqrt{g^s})$ is well-defined on $\Omega$ and that $g_0(z) = \sqrt{g^s}(z)$ for every zero $z$ of $g_v^s,$ because $-1$ is not in the image of  $g_0/\sqrt{g^s}.$ Hence the logarithm is given by  Formula
(\ref{case4i}).\qed

 \begin{remark} Notice that in the hypotheses of case (c) of Theorem~\ref{general}, the stated sufficient conditions are always fulfilled on  ``small" basic product domains that are neighbourhoods of  a (non real) 
   $z_0 \in Z(\omega)$ (For instance on any set $\bS B^4(z_0,r)$ with small enough $r> 0$).
   \end{remark}

Since, by definition, every set $\mathcal{SR}_{\omega}(\Omega)$ contains also the set $\mathcal{SR}_{\R}(\Omega),$  Theorem \ref{general} $(b)$ yields  the following
\begin{corollary} Let $\Omega$ be a basic domain, $g \in \mathcal{SR}_{\R}(\Omega)$ and let $\omega$ be a vectorial class in $\Omega$ with
  $Z(\omega) \cap \Omega = \varnothing.$ Then there exists a logarithmic function of $g$
  in the class $\mathcal{SR}_{\omega}(\Omega),$ denoted by $\log_*g.$ 
\end{corollary}

\subsection{The case of one isolated non real zero} \label{6.2}

For the case of a slice--regular function defined on a basic product domain, and whose vectorial class has only an isolated zero, we can - as announced - produce a formula for the $*$-logaritms.

\begin{theo}\label{isolated} Let $g \in \cS\mathcal{R}_{\omega}({\Omega})$ be a nonvanishing function and $\Omega$ be a basic product domain. Let   $Z(\omega) \cap \Omega =\{z_0\}$ and let $\sqrt{g^s}$ be such that $\sqrt{g^s}(z_0) = g_0(z_0).$  Then there exist a logarithmic function $f$ of $g$, $f \in \mathcal{SR}_{\omega}(\Omega),$  given by the formula
\begin{equation}
  f=\log_{*,m,0 \cdot[g_v]} g =  \log_{*,m,0} ((-1)^m \sqrt{{g^s}}) +  \frac{g_v}{\nu\left(\mu^{-1}\left(g_0/ \sqrt{g^s}\right) \right)\sqrt{g^s}},
\end{equation}
 where  $\log_{*,m,0}=\log_{*,0,0}  + m\pi \mathcal{I}$ for $m \in \Z$ and
\begin{enumerate}
\item[(a)]  $m$ is even if $ \sqrt{g^s}(\bS_{z_0})\subset \H \setminus (-\infty,0)$ or odd if $\sqrt{g^s}(\bS_{z_0})\subset  (-\infty,0)$,
\item[(b)] the function \[\mu^{-1}\left(g_0/ \sqrt{g^s}\right)\]
 is the lift of the function $g_0/ \sqrt{g^s}$ with respect to the mapping $\mu$ such that
\[\mu^{-1}\left(g_0/ \sqrt{g^s}\right) (z_0) = 0.\]
\end{enumerate}
\end{theo}

\begin{proof}
The logarithm of $\sqrt{g^s}$ exists by Theorem \ref{general} $(a)$, because $\Omega \subset \H \setminus \R.$ The  function
$\sqrt{g^s}$ is such that not only $\sqrt{g^s}(z_0)=g_0(z_0)$ but also  $\sqrt{g^s}(z)=g_0(z)$ for every $z \in \bS_{z_0}.$ Indeed,
on $\bS_{z_0}$ we have $\sqrt{g^s}(z)=\sqrt{g_0^2}(z)$ and since  $\sqrt{g_0^2}(z_0) = g_0(z_0),$ the same holds on the whole sphere $\bS_{z_0}.$
We have to show that the lift of  the function  $g_0/\sqrt{g^s}$ via $\mu$ can be defined on $\Omega.$ First observe that $(g_0/\sqrt{g^s})^{-1}(1) \cap \Omega = \bS_{z_0}.$ Let $I$ be such that $z_0 \in \C_{I,+}$ and choose an arc $l_{I,+}$ connecting $z_0$ to the boundary of 
$\Omega_{I,+}:=\Omega \cap \C_{I,+}.$ Let $l_{I,-}$ be the reflected arc. The domain $\Omega_{I,+} \setminus l_{I,+}$ is simply connected and
$$
  (g_0/\sqrt{g^s})(\Omega_{I,+} \setminus l_{I,+}) \subset \C_I \setminus \{\pm 1\}.
$$
Since the map $\mu$ is a slice--covering map from $\H \setminus \{k^2 \pi^2, k \in \bN\}$ to $\H\setminus \{\pm 1\},$ the lift $G$ of $g_0/\sqrt{g^s}$  exists,
$$\xymatrix
   {  & \C_I \setminus \{k^2 \pi^2, k \in \bN\} \ar[dr]^{\mu} &  \\
      \Omega_{I,+}  \ar[ur]^{G} \ar[rr]^{g_0/\sqrt{g^s}}&    & \C_I \setminus \{\pm 1\}
   }
$$
and can be chosen in such a way that $\lim_{z \ra z_0} G (z) = 0.$
Cover the arc $l_{I,+}$ with a (possibly infinite) chain of discs $D_i$ such that $z_0 \in D_0$ and each $D_i$ intersects only $D_{i-1}$ and $D_{i+1}$ and intersections are connected.
Let $G_0:= \mu_0^{-1}(g_0/\sqrt{g^s})$ near $z_0.$ The lifts $G$ and $G_0$ coincide on $D_0 \setminus l_{I,+}$ and hence define
a lift on the union $D_0 \cup (\Omega_{I,+} \setminus l_{I,+})$ which we also denote by $G.$ Since $z_0$ is the only point with value $g_0/\sqrt{g^s}(z) = 1$ in $\Omega_{I,+},$ we can  choose a lift $G_1$ on $D_1$ so that it matches $G_0$ on $D_0 \cap D_1.$ Since $D_1 \cap (D_0 \cup \Omega_{I,+})$ is connected, by the lifting property, the lift $G_1$ also matches the lift $G$ on $D_1 \setminus l_{I,+}.$ Repeating this procedure extends the lift $G$ to $\Omega_{I,+}.$ Notice that $G(z) \ne k^2\pi^2, k \in \bN$ and so $\nu(G) \ne 0.$  Extend the definition of $G$ to $\Omega_{I,-}$ by Schwarz reflection and then use the extension formula to get a slice--preserving function, which we denote - with a slight abuse of notation - by $\mu^{-1}(g_0/\sqrt{g^s}).$
The logarithm is now given by  Formula
(\ref{case4i}).\end{proof}

\begin{example}{\em
Consider the function
$$
  g(z) = -1 + z^2i + \sqrt{2} z j + k
$$ defined on $\mathbb{H}.$ Because $g^s = 1 + (z^2 + 1)^2,$ the
  zeroes of $g^s$ on $\C_I$ are $z_{1,2} = \sqrt{\pm I - 1}$ and $z_{3,4} = -\sqrt{\pm I - 1}$.
  Hence these zeroes
  lie on the sphere with radius $r = 2^{1/4} > 1.1,$ so $g$
  is nonvanishing in the ball $ \Omega = B^4(0,1.1).$
  A simple
  calculation shows that $g^s$ maps $\Omega_I$ to a cardioid-shaped
  domain in the right half-plane of $\H,$ so the image misses the
  negative real axis, hence there exists a unique logarithmic function
  of $g^s,$ namely $\log_0g^s.$ Since the domain $\Omega$ intersects
  the real axis, the necessary condition for logarithm of $\sqrt{g^s}$
  to exist is (\ref{cond1}), $\sqrt{g^s}(\Omega \cap \R) \subset
  (0,+\infty),$ therefore the only possibility for the definition of
  $f_0$ is to take the principal branch of the square root and set
\[f_0:=\log_{*,0,0}\sqrt{g^s}=\dfrac{1}{2} \log_{*,0,0}{g^s}.\]
 The vectorial function $g_v$ has only $z_0=\frac{k-i}{\sqrt 2} j$  as the unique (double) zero and  the symmetrization of $g_v$ is
 $g_v^s = (z^2 + 1)^2$ on $\Omega.$
Unfortunately, $g(z_0) =g_0(z_0)= -1$ and $\sqrt{g^s}(z_0) = 1$ and hence condition (\ref{counterex}) no longer holds, which makes it impossible to define
the functions $f_v$ and $f_v^s$ near the point $z_0,$ because $-1$ is a branching point for $\mu^{-1}.$
Notice that the function $g$  meets the hypotheses of Theorem \ref{general} $(c),$ but does not fulfil one of the stated sufficient conditions,
namely condition (\ref{counterex}).
  
}
\end{example}

\end{document}